\documentclass[11pt, reqno]{amsart}
\usepackage[dvipsnames]{xcolor}
\usepackage[colorlinks=true,citecolor=Aquamarine, linkcolor=DarkOrchid]{hyperref}
\usepackage{csquotes}
\usepackage{fourier}
\usepackage[T1]{fontenc}
\usepackage[shortlabels]{enumitem}
\usepackage{verbatim}
\usepackage{calligra}
\usepackage{geometry}
\usepackage{amsmath,amssymb,amsfonts,amsthm}
\usepackage[all]{xy}

\usepackage{tikz-cd}

\theoremstyle{plain}
\newtheorem{thm}{Theorem}[section]
\newtheorem*{mainthm*}{Main theorem}
\newtheorem{prop}[thm]{Proposition}
\newtheorem{lem}[thm]{Lemma}
\newtheorem{coro}[thm]{Corollary}

\theoremstyle{definition}
\newtheorem{defi}[thm]{Definition}

\theoremstyle{remark}
\newtheorem{rem}[thm]{Remark}

\renewcommand{\leq}{\leqslant}
\renewcommand{\geq}{\geqslant}
\newcommand{\F}{\mathcal{F}}
\newcommand{\G}{\mathcal{G}}
\newcommand{\rk}{\operatorname{rk}}
\newcommand{\DM}{\mathbb{D}}

\newcommand{\DA}{\mathbb{DA}}
\newcommand{\D}{\operatorname{D}}
\newcommand{\ct}{\operatorname{ct}}

\DeclareMathOperator{\id}{id}

\def\AC{\mathcal{A}}

\def\DM{\mathbb{D}}
\def\KM{\mathbb{K}}
\def\QM{\mathbb{Q}}

\newcommand{\KMMod}{\operatorname{\mathbb{K}-Mod}}
\newcommand{\Sch}{\operatorname{Sch}}

\DeclareMathOperator{\Id}{Id}

\DeclareMathOperator{\hocolim}{hocolim}
\DeclareMathOperator{\Rlim}{R\,lim}
\DeclareMathOperator{\holim}{holim}
\let\lim\relax
\DeclareMathOperator{\lim}{lim}
\DeclareMathOperator{\Hom}{Hom}

\title{The localization spectral sequence in the motivic setting}
\date{}

\author[C. Dupont]{Cl\'{e}ment Dupont}
\address{Institut Montpelli\'erain Alexander Grothendieck, Universit\'{e} de Montpellier, CNRS, Montpellier, France}
\email{clement.dupont@umontpellier.fr}

\author[D. Juteau]{Daniel Juteau}
\address{LAMFA, Universit\'e de Picardie Jules Verne, CNRS, 33 rue Saint-Leu, 80000 Amiens, France}
\email{daniel.juteau@u-picardie.fr}

\begin{document}

\maketitle

\begin{abstract}
We construct and study a motivic lift of a spectral sequence associated to a stratified scheme, recently discovered by Petersen in the context of mixed Hodge theory and $\ell$-adic Galois representations. The original spectral sequence expresses the compactly supported cohomology of an open stratum in terms of the compactly supported cohomology of the closures of strata and the combinatorics of the poset underlying the stratification. Some of its special cases are classical tools in the study of arrangements of subvarieties and configuration spaces. Our motivic lift lives in the triangulated category of \'{e}tale motives and takes the shape of a Postnikov system. We describe its connecting morphisms and study some of its functoriality properties.
\end{abstract}

\section*{Introduction}

    For a topological space $X$, an open subspace $U$ and a complementary closed subspace $Z$, the compactly supported cohomology groups of $X$, $U$, $Z$ are related by a localization long exact sequence:
    \begin{equation}\label{eq: localization long exact sequence}
    \cdots \longrightarrow  H^\bullet_c(U)\longrightarrow H^\bullet_c(X) \longrightarrow H^\bullet_c(Z)\longrightarrow H^{\bullet+1}_c(U)\longrightarrow \cdots
    \end{equation}
    This can typically be used for two different purposes: either to compute the compactly supported cohomology of $X$ knowing that of $U$ and $Z$, or to compute the compactly supported cohomology of $U$ knowing that of $X$ and $Z$.
    
    More generally, let $X$ be a topological space equipped with a stratification, i.e., a partition by locally closed subspaces called strata such that the closure of a stratum is a union of strata; we assume for simplicity that there is a unique open stratum $X_0$. The specialization relation turns the set of strata into a finite poset whose least element is $X_0$. One may either want to understand the space $X$ in terms of the strata, or to understand the open stratum $X_0$ in terms of the closures of the strata. In the former case, the localization long exact sequence can be generalized to a spectral sequence in an obvious way. In the latter case, however, this was explained only recently by Petersen \cite{petersen} who devised a spectral sequence converging to the compactly supported cohomology of $X_0$, whose first page is expressed in terms of the compactly supported cohomology of the closures of strata, and of the combinatorics of the poset of strata. We refer the reader to the introduction of [\emph{loc. cit.}] for a clear interpretation in terms of inclusion-exclusion. 
    
    A precursor of Petersen's spectral sequence (or rather, of its Poincar\'{e} dual version) is Deligne's spectral sequence appearing in mixed Hodge theory \cite[3.2.4.1]{delignehodgeII} where the stratification is induced by a normal crossing divisor inside a smooth projective complex variety. Several other special cases are classical tools in the study of more combinatorially involved contexts such as arrangements of subvarieties \cite{goreskymacpherson, looijenga, bjornerekedahl, dupontOS, bibby} and configuration spaces \cite{cohentaylor, kriz, totaro, getzler}. In the general case, Petersen proves that his spectral sequence is compatible with mixed Hodge structures when $X$ is a complex algebraic variety equipped with an algebraic stratification. It also has an \'{e}tale $\ell$-adic variant which is compatible with Galois actions. The proofs are sheaf-theoretic and  involve filtering well-chosen resolutions in abelian categories of sheaves.
    
    The goal of this article is to lift Petersen's spectral sequence to a motivic setting. Let now $X$ be a scheme equipped with a stratification (see \S\ref{sec: sec main theorem} for the relevant assumptions) with a unique open stratum $X_0$, and let $j:X_0\hookrightarrow X$ denote the open immersion. We also denote by $i_{\overline{S}}^X:\overline{S}\hookrightarrow X$ the closed immersion of the closure of a stratum $S$. We denote by $\hat{P}$ the poset of strata and fix a strictly increasing map $\sigma:\hat{P}\to \mathbb{Z}$ such that $\sigma(X_0)=0$. We fix a ring of coefficients $\KM$. To every stratum $S\in \hat{P}$ is associated a cochain complex of $\KM$-modules $C^\bullet(S)$ which computes the reduced cohomology of the open interval $(X_0,S)$ in the poset $\hat{P}$.
    
    We work in the context of the triangulated category of \'{e}tale motives (or motivic sheaves) over $X$ with coefficients in $\KM$, denoted $\DA_X$ \cite{ayoubPhD1, ayoubPhD2, ayoubguide, cisinskidegliseetale, cisinskideglise}. The lack of an abelian-categorical formalism for motivic sheaves forces us to depart from Petersen's original techniques.
    In the triangulated setting, the notion of a filtration has to be replaced with that of a Postnikov system, that is, a sequence of distinguished triangles where each triangle has a vertex in common with the next one. The main result of this article is as follows (see Theorem \ref{thm:maintheorem} and Theorem \ref{thm: functoriality} for more precise statements).
    
    \begin{mainthm*}
    For $\mathcal{F}\in \DA_X$ there is a Postnikov system in $\DA_X$:
	$$\xymatrix{
	\cdots &\ar[rr]&&F^2\ar[rr]\ar[ld] && F^1\ar[rr]\ar[ld]  && F^0= j_!j^!\F \ar[ld]  \\
	&&G^2\ar[lu]^{+1}&& G^1\ar[ul]^{+1} && G^0 \ar[ul]^{+1}& 
	}$$
	where the graded objects are given by
	$$G^k = \bigoplus_{\substack{S\in\hat{P}\\ \sigma(S)=k}} (i_{\overline{S}}^X)_*(i_{\overline{S}}^X)^*\F \otimes C^\bullet(S) \ .$$
	The connecting morphisms $G^k\to G^{k+1}[1]$ are explicitly computed. This Postnikov system is functorial in $\F$ and functorial with respect to a class of stratified morphisms.
    \end{mainthm*}
    
    In the case of the constant motivic sheaf $\F=\KM_X$, this theorem expresses the compactly supported motive of $X_0$ in terms of the compactly supported motives of the closures of strata $\overline{S}$ and the complexes $C^\bullet(S)$. For instance, if the stratification consists of an open $j:U\hookrightarrow X$ and its closed complement $i:Z\hookrightarrow X$, the Postnikov system reduces to the localization triangle
    $$j_!\KM_U\longrightarrow \KM_X \longrightarrow i_*\KM_Z\stackrel{+1}{\longrightarrow}$$
    which is the motivic lift of the localization long exact sequence \eqref{eq: localization long exact sequence}.
    
    One can recover Petersen's spectral sequence(s) along with a description of the $d_1$ differential from our main theorem, by applying (compactly supported) cohomological realization functors. In a genuinely motivic setting, an application to the study of classical polylogarithm motives will appear as a joint article of the first author with J. Fres\'{a}n \cite{dupontfresan}. There, it is crucial to have a Postnikov system that is functorial with respect to a group action on a stratified scheme, which is a special case of the functoriality statement in our theorem.
    
    One of the main difficulties in the proof of our main theorem is to construct the Postnikov system in a way that makes it obviously functorial. For this we cannot simply work in the context of a triangulated category, where cones are not functorial. Rather, we are led to work in the enhanced setting of triangulated derivators. Another reason for this choice is that we rely on the six functor formalism for \'{e}tale motives, developed by Ayoub in \cite{ayoubPhD1,ayoubPhD2} and written in the language of algebraic derivators, a geometric enrichment of the notion of a triangulated derivator. From the standpoint of homotopy theory, it is natural to expect our main theorem to lift to the stable $\infty$-category of motives; this would require an $\infty$-categorical enhancement of Ayoub's six functor formalism.
    
    We also study a dual version of our main theorem (Theorem \ref{thm: main theorem dual}) where we are interested in describing the object $j_*j^*\F$. Due to the lack of duality in the general setting of algebraic derivators, we cannot simply repeat the proof. Instead,
    we rely on applying Poincar\'{e}--Verdier duality, but the latter is available at the motivic level only under certain assumptions (see \S \ref{subsec: dual version}). Note that, if we gave up on functoriality, then we would not need to work in the setting of algebraic derivators  and could prove the dual statement (without functoriality) in full generality. This strongly suggests that the dual statement (with functoriality) is true in full generality, even though we are not able to prove it with our methods. In any case, if one is only interested in working with realizations, one can first apply a realization functor to the main theorem and then dualize.
    
    \subsection*{Perspectives}
    
        A natural direction of research would be to try and apply our main theorem to prove motivic representation stability results in the spirit of the homological representation stability results of Petersen \cite{petersen}. Also, it would be desirable to clarify the general functoriality properties of our construction, beyond those already explored here. 
        
        A motivation for this project is the possibility to study a more general geometric setting mixing $j_!$ and $j_*$ extensions, depending on the strata. The corresponding motives can be viewed as relative cohomology motives on some blow-up of the ambient variety and are ubiquitous in the geometric study of periods (see, e.g., \cite{goncharovperiods} and the introduction of \cite{dupontbi}).

    \subsection*{Outline}

        In \S \ref{sec: 1} we review classical definitions and properties of poset (co)homology; to the best of our knowledge, the only original content is the introduction of connecting morphisms relating poset (co)homology complexes of different intervals in a poset. In \S \ref{sec: 2} we work in the setting of triangulated derivators and collect some tools to produce and study functorial Postnikov systems. In \S \ref{sec: sec main theorem} we apply those tools to our geometric setting and prove the main results.

    \subsection*{Acknowledgements}
    This work was partially written at the Max Planck Institute for Mathematics and the Hausdorff Institute for Mathematics in Bonn and we would like to thank these institutes for their hospitality and the excellent working conditions there. We also gratefully acknowledge support from the ANR grants PERGAMO (ANR-18-CE40-0017) and GEREPMOD (ANR-16-CE40-0010-01). 
    
    Many thanks to Joseph Ayoub, Martin Gallauer, Georges Maltsiniotis and Simon Pepin Lehal\-leur for stimulating conversations and clarifications on derivators and categories of motives.

\section{Poset (co)homology}\label{sec: 1}

    In this section we review poset (co)homology. To the best of our knowledge, the only original content is the introduction of connecting morphisms relating poset (co)homology complexes of different intervals in a poset. We fix a commutative ring with unit $\KM$ for the rest of this article, that will serve as a ring of coefficients.
    
	\subsection{Definition}
	\label{subsec:def C}
	
	    Let $P$ be a finite poset. We will sometimes make use of the extension $\widehat{P}=\{\hat{0}\}\cup P$ where $\hat{0}<p$ for all $p\in P$. For any element $x\in P$ we let $C_\bullet(x)$, denoted $C_\bullet^P(x)$ when we want to make dependence on $P$ explicit, be the chain complex whose degree $n$ component is the free $\mathbb{K}$-module on chains
		\begin{equation*}
		\left[x_1<\cdots<x_{n-1}<x_n=x\right]\ ,
		\end{equation*}
		and whose differential $\partial:C_n(x)\rightarrow C_{n-1}(x)$ is given by
		$$\partial\left[x_1<\cdots<x_{n-1}<x_n=x\right] = \sum_{i=1}^{n-1} (-1)^{i-1} \left[x_1<\cdots<\widehat{x_i}<\cdots <x_{n-1}<x_n=x\right]\ .$$
		 We let $h_\bullet(x)$ denote the homology of $C_\bullet(x)$. Up to a shift, $C_\bullet(x)$ is the (reduced) normalized chain complex of the nerve of the poset $P_{<x}=\{p\in P\; ,\; p<x\}$ and thus we have
		$$h_n(x)=H_n(C_\bullet(x))=\widetilde{H}_{n-2}(P_{<x})\ .$$
		We let $C^\bullet(x)$, or $C^\bullet_P(x)$ when we want to make dependence on $P$ explicit, denote the cochain complex dual to $C_\bullet(x)$ and use the same notation for the basis of chains and the (dual) basis of cochains. The differential $d:C^n(x)\rightarrow C^{n+1}(x)$ is given by
		$$d\left[ x_1<\cdots <x_{n-1}<x_n=x\right] = \sum_{i=1}^n (-1)^{i-1}\sum_{x_{i-1}<y<x_i}\left[ x_1<\cdots <x_{i-1}<y<x_i<\cdots <x_{n-1}<x_n=x \right]\ ,$$
		where by convention we have $x_0=\hat{0}$ in $\widehat{P}$. We let $h^\bullet(x)$ denote the cohomology of $C^\bullet(x)$ and we have:
		$$h^n(x)=H^n(C^\bullet(x))=\widetilde{H}^{n-2}(P_{<x})\ .$$
	    The following lemma is classical.
		
		\begin{lem}\label{lem:C is contractible}
		If $P$ has a least element $a$ then $C_\bullet(x)$ and $C^\bullet(x)$ are contractible for all $x>a$.
		\end{lem}
	
		\begin{proof}
		The nerve of $P_{<x}=[a,x)$ is a cone over the nerve of the open interval $(a,x)$ and thus contractible. Concretely, a contracting homotopy $c:C_\bullet(x)\rightarrow C_{\bullet+1}(x)$ is provided by the formula:
		$$c\left[x_1<\cdots <x_{n-1}<x_n= x\right]=\begin{cases}  0 & \textnormal{if } x_1=a; \\ \left[a<x_1<\cdots<x_{n-1}<x_n= x\right] & \textnormal{if } x_1>a. \end{cases}$$
		The transpose of $c$ is a contracting homotopy for $C^\bullet(x)$.
		\end{proof}
		
		It is sometimes convenient to extend the definitions to $\hat{P}$ by defining $C_\bullet(\hat{0})$ and $C^\bullet(\hat{0})$ to be $\KM$ concentrated in degree zero.
		
		\begin{rem}\label{rem: functoriality complexes C}
		The complexes $C_\bullet$ have a certain functoriality property with respect to morphisms of posets. In this article we will only deal with functoriality with respect to isomorphisms (and in particular with respect to group actions). For $\alpha:P\to P'$ an isomorphism of posets we have for every $x\in P$ a natural isomorphism of chain complexes $C_\bullet(\alpha):C_\bullet^P(x)\to C_\bullet^{P'}(x')$ for $x'=\alpha(x)$. They satisfy $C_\bullet(\id)=\id$ and $C_\bullet(\beta\circ\alpha)=C_\bullet(\beta)\circ C_\bullet(\alpha)$. Dually we have natural isomorphisms of cochain complexes $C^\bullet(\alpha):C^\bullet_{P'}(x')\to C^\bullet_P(x)$ that satisfy $C^\bullet(\id)=\id$ and $C^\bullet(\beta\circ\alpha)=C^\bullet(\alpha)\circ C^\bullet(\beta)$.
		\end{rem}
		
	\subsection{The connecting maps}\label{subsec: connecting morphisms}
	
		For $x<y$ in $P$ we define a map
		$$b_x^y: C_{\bullet+1}(y) \rightarrow C_{\bullet}(x)$$
		by setting 
		$$b_x^y \left[x_1<\cdots<x_{n}<x_{n+1}=y\right]=\begin{cases} (-1)^{n}\left[x_1<\cdots <x_{n}=x\right] & \textnormal{if } x_{n}=x; \\ 0 & \textnormal{otherwise.}\end{cases}$$
	
		\begin{lem}\label{lem: commutator of b and d}
		We have 
		$$(\partial b_x^y + b_x^y\partial)\left[x_1<\cdots < x_{n}<x_{n+1}=y\right] = \begin{cases}  \left[x_1<\cdots <x_{n-1}=x\right] & \textnormal{if } x_{n-1}=x; \\ 0 & \textnormal{otherwise.}\end{cases}$$
		\end{lem}
		
		\begin{proof}
		We compute, for $X=\left[x_1<\cdots <x_{n}<x_{n+1}=y\right]$:
		$$ b_x^y\partial X  =  \sum_{i=1}^{n-1}(-1)^{i-1}b_x^y\left[x_1<\cdots<\widehat{x_i}<\cdots<x_{n}<x_{n+1}=y\right] +(-1)^{n-1} b_x^y\left[x_1<\cdots <x_{n-1}<x_{n+1}=y\right]\ .$$
		If $x_{n-1}=x$ then $x_{n}\neq x$ and we have $b_x^y\partial X=\left[x_1<\cdots <x_{n-1}=x\right]$ and $\partial b_x^y X=\partial 0=0$, which proves the first part of the claim.
		If $x_{n-1}\neq x$ and $x_{n}\neq x$ then all terms vanish and we get $b_x^y\partial X=\partial b_x^y X=0$. If $x_{n-1}\neq x$ and $x_{n}=x$ then we get
		$$b_x^y\partial X = \sum_{i=1}^{n-1}(-1)^{n-i}\left[x_1<\cdots<\widehat{x_i}<\cdots<x_{n}=x\right] = - \partial b_x^y X\ ,$$
		which finishes the proof.
		\end{proof}
		
		We write $x\lessdot y$ when $y$ covers $x$ in $P$, i.e., when $x<y$ and there is no $z\in P$ such that $x<z<y$.
		
		\begin{lem}\label{lem: b morphism homology}
		\phantom{}
		\begin{enumerate}[1)]
		\item For $x\lessdot y$ in $P$, $b_x^y:C_{\bullet+1}(y)\rightarrow C_{\bullet}(x)$ is a morphism of complexes.
		\item Let $x<z$ in $P$ such that every $y\in (x,z)$ satisfies $x\lessdot y \lessdot z$. Then the morphism of complexes
		$$\sum_{x<y<z}b_x^y b_y^z : C_{\bullet+2}(z)\rightarrow C_{\bullet}(x)$$
		is homotopic to zero.
		\end{enumerate}
		\end{lem}
		
		The first part of the lemma implies that we get connecting morphisms $b_x^y:h_{\bullet+1}(y)\to h_\bullet(x)$ in homology, for $x\lessdot y$.
		
		\begin{proof}
		\begin{enumerate}[1)]
		\item For $x_{n-1}<x_{n}<x_{n+1}=y$ we cannot have $x_{n-1}=x$ since $y$ covers $x$. Then Lemma \ref{lem: commutator of b and d} implies that $\partial b_x^y = - b_x^y\partial$, thus $b_x^y$ is a morphism of complexes.
		\item We have
		$$\sum_{x<y<z}b_x^yb_y^z\left[x_1<\cdots <x_{n+1}<x_{n+2}=z\right]=\begin{cases}  - \left[x_1<\cdots <x_{n}=x\right] & \textnormal{if } x_{n}=x; \\ 0 & \textnormal{otherwise.}\end{cases}$$
		Thanks to Lemma \ref{lem: commutator of b and d} this can be rewritten as
		$$\sum_{x<y<z}b_x^y b_y^z=-\partial b_x^z-b_x^z\partial\ ,$$
		which proves the claim.
		\end{enumerate}
		\end{proof}
		
		By duality we get a map that we denote by the same symbol, since there is no risk of confusion:
		$$b_x^y:C^{\bullet}(x)\rightarrow C^{\bullet+1}(y)\ .$$
		It is defined by the formula
		$$b_x^y\left[x_1<\cdots<x_{n}=x\right] = (-1)^{n} \left[x_1<\cdots<x_{n}=x<x_{n+1}=y\right]\ .$$

		\begin{lem}\label{lem: b morphism cohomology}
		\phantom{}
		\begin{enumerate}[1)]
		\item For $x\lessdot y$ in $P$, $b_x^y:C^{\bullet}(x)\rightarrow C^{\bullet+1}(y)$ is a morphism of complexes.
		\item Let $x<z$ in $P$ such that every $y\in (x,z)$ satisfies $x\lessdot y \lessdot z$. Then the morphism of complexes
		$$\sum_{x<y<z}b_y^z b_x^y : C^{\bullet}(x)\rightarrow C^{\bullet+2}(z)$$
		is homotopic to zero.
		\end{enumerate}
		\end{lem}
		
		\begin{proof}
		This is the dual of Lemma \ref{lem: b morphism homology}.
		\end{proof}
		
		It is sometimes convenient to extend the definitions to $\hat{P}$. Indeed, for $\hat{0}\lessdot y$, i.e., for $y$ a minimal element of $P$, we can set $b_{\hat{0}}^y:C_{\bullet+1}(y)\to C_\bullet(\hat{0})$ to be the natural (iso)morphism of complexes. The same goes in cohomology for $b_{\hat{0}}^y:C^\bullet(\hat{0})\to C^{\bullet+1}(y)$. One easily checks that Lemma \ref{lem: b morphism homology} and Lemma \ref{lem: b morphism cohomology} also apply to the case $x=\hat{0}$.
		
		\begin{rem}\label{rem: graded poset OS}
		Let us assume for simplicity that the poset $\hat{P}$ is graded, i.e., any two maximal chains between two elements $x<y$ in $\hat{P}$ have the same length. For $x\in \hat{P}$ let $\rk(x)$ denote the length of a maximal chain from $\hat{0}$ to $x$. In many geometric cases we have, for every $x\in\hat{P}$:
		$$h_n(x)=0 \quad \mbox{ for }\quad  n\neq \rk(x)\ ,$$
		and we simply write $h(x)=h_{\rk(x)}(x)$. (This implies that the cohomology of $C^\bullet(x)$ is concentrated in degree $\rk(x)$ and that we have $h^{\rk(x)}(x)\simeq h(x)^\vee$.) This condition is satisfied, e.g., if the poset $\hat{P}$ is Cohen--Macaulay \cite{baclawskicohenmacaulay, bjornergarsiastanley}. In this case we get a chain complex $(h,b)$ where
		$$h_n=\bigoplus_{\substack{x\in \hat{P} \\ \rk(x)=n}}h(x)$$
		and $b:h_{n+1}\to h_n$ is induced by the connecting maps $b_x^y$ for $x<y$. One can also prove that these connecting maps induce acyclic complexes of $\KM$-modules, for every $x\in P$:
		$$0\longrightarrow h(x) \longrightarrow \bigoplus_{\substack{y\in\hat{P}, y<x \\ \rk(y)=\rk(x)-1}} h(y) \longrightarrow \bigoplus_{\substack{z\in\hat{P}, z<x \\ \rk(z)=\rk(x)-2}} h(z) \longrightarrow \cdots \longrightarrow h(\hat{0})\longrightarrow 0 \ .$$
		This allows one to define $h(x)$ together with the connecting morphisms $b_u^x$ by induction on $\rk(x)$.
		
		A typical example of a Cohen--Macaulay poset is the poset of flats of a matroid (for instance, the poset of strata of a central hyperplane arrangement); in this case $(h,b)$ is the underlying chain complex of the Orlik--Solomon algebra of the matroid \cite{orliksolomon, orlikteraobook}.
		\end{rem}

	\subsection{Interpretation of poset cohomology as homotopy limit}\label{subsec: poset cohomology holim}
	
	We will now consider the abelian category of representations of the finite poset $P$, i.e., the category $(\KMMod)^P$ of functors from $P$ viewed as a category to the category of $\KM$-modules. Since $\KMMod$ is abelian, it admits finite limits, so we have a limit functor $\lim_P : (\KMMod)^P \to \KMMod$, which is right adjoint to the constant functor $\KMMod \to (\KMMod)^P$; since it has a left adjoint, it is left exact, and we may consider the right derived functor $\Rlim_P : D((\KMMod)^P) \to D(\KMMod)$. In anticipation of the next section, we will call it homotopy limit and denote it by $\holim_P$. We now prove and discuss the following interpretation of the complexes $C^\bullet(x)$ (see also \cite{tosteson} for a similar discussion).
	
		\begin{prop}\label{prop:Cx as holim}
	    For $x\in P$ we denote by $\KM_x$ the representation of $P$ defined by $\KM_x(y)=\KM$ if $y=x$ and zero otherwise. We have a canonical isomorphism in $D(\KMMod)$:
	    $$
	    \holim_P\KM_x \simeq C^{\bullet+1}(x)\ .
	    $$
	    \end{prop}

	    In order to compute the functor $\holim_P$ we introduce convenient $\lim_P$-acyclic representations of $P$. For $x\in P$ and $M\in \KMMod$, we let $M_{\leq x}\in(\KMMod)^P$ denote the representation given by $M_{\leq x}(y)=M$ if $y \leq x$ and zero otherwise, the transition morphisms being the identity of $M$ or zero. 
	    
	    \begin{lem}\label{lem: Mleqx is limPacyclic}
	    The representation $M_{\leq x}$ is $\lim_P$-acyclic.
	    \end{lem}
	    
	    \begin{proof}
	    The functor
	    $$(-)_{\leq x}:\KMMod \longrightarrow (\KMMod)^P\; , \; M\mapsto M_{\leq x}$$
	    is exact and sends injectives to injectives. Indeed, for $T\in (\KMMod)^P$ we have an isomorphism
	    $$\Hom_{(\KMMod)^P}(T,M_{\leq x}) \simeq \Hom_{\KMMod}(T(x),M)\ .$$
	    and thus the functor $\Hom_{(\KMMod)^P}(-,M_{\leq x})$ is exact if $M$ is injective. Thus, we have isomorphisms:
	    $$\Rlim_P(M_{\leq x}) \simeq \Rlim_P\circ R(-)_{\leq x}(M) \simeq R(\lim_P\circ (-)_{\leq x})(M) \simeq M \simeq \lim_P(M_{\leq x})\ .$$
	    The first isomorphism follows from the fact that $(-)_{\leq x}$ is exact, the second follows from the fact that it sends injectives to injectives, the third and fourth from the equality $\lim_P\circ (-)_{\leq x}=\Id_{\KMMod}$. The claim follows.
	    \end{proof}

	    \begin{proof}[Proof of Proposition \ref{prop:Cx as holim}]
        For $z\leq y$ we have a canonical morphism $\KM_{\leq y}\rightarrow \KM_{\leq z}$. Moreover, those morphisms compose functorially. Using them we can form a resolution
	    $$0\rightarrow \KM_x \rightarrow \KM_{\leq x}\rightarrow \bigoplus_{y<x} \KM_{\leq y} \rightarrow \bigoplus_{z<y<x} \KM_{\leq z}\rightarrow \cdots$$
	    More precisely we set
	    $$R^n_x = \bigoplus_{[x_1<\cdots <x_n<x_{n+1}=x]} \KM_{\leq x_1}\ .$$
	    In analogy with the construction of the complexes $C^\bullet(x)$ of \S \ref{subsec:def C}, we define a differential $d:R^n_x\rightarrow R^{n+1}_x$. Its component indexed by chains $[x_1<\cdots <x_n<x_{n+1}=x]$ on the source and $[x_1<\cdots<x_{i-1}<y<x_i<\cdots <x_n<x_{n+1}=x]$ and on the target equals $(-1)^i$ times the natural map (the latter being the identity for $i>1$ and the canonical morphism $\KM_{\leq x_1}\rightarrow \KM_{\leq y}$ for $i=1$). The other components are zero. One easily checks that we get a complex $R^\bullet_x$ of representations of $P$ which is such that 
	    $$R^\bullet_x(a)=\begin{cases} \KM & \hbox{ if } a=x; \\ C^{\bullet+1}_{[a,x]}(x) & \hbox{ if } a<x; \\ 0 & \hbox{ otherwise.}\end{cases}$$
	    By Lemma \ref{lem:C is contractible}, the complex $C^\bullet_{[a,x]}(x)$ is contractible and we thus get a resolution $\KM_x\stackrel{\sim}{\rightarrow}R^\bullet_x$.
	    
	    By Lemma \ref{lem: Mleqx is limPacyclic} this resolution is $\lim_P$-acyclic. Hence, it can be used to compute $\holim_P \KM_x = \Rlim_P \KM_x$. Since each $\lim_P \KM_{\leq x_1}$ is just $\KM$, applying $\lim_P$ to the resolution gives $\lim_PR^\bullet_x\simeq C^{\bullet+1}(x)$, and the result follows.
	    \end{proof}
	    
	    \begin{rem}
	    The resolution appearing in the proof of Proposition \ref{prop:Cx as holim} is a Bousfield--Kan resolution \cite[Chapter XI]{bousfieldkan}.
	    \end{rem}
	    
	    We now turn to the interpretation of the connecting morphisms $b_x^y$. For $x<y$ in $P$ we let $\KM_x^y$ denote the representation of $P$ defined by $\KM_x^y(z)=\KM$ if $z\in \{x,y\}$ and zero otherwise, the transition morphism $\KM_x^y(x)\rightarrow \KM_x^y(y)$ being the identity. We have a short exact sequence in $(\KMMod)^P$:
	    $$0\longrightarrow \KM_y\longrightarrow \KM_x^y \longrightarrow \KM_x\longrightarrow 0\ ,$$
	    which induces a distinguished triangle $\KM_y\longrightarrow \KM^y_x \longrightarrow \KM_x \stackrel{+1}{\longrightarrow}\ $ in $D((\KMMod)^P)$. We denote by 
	    $$a_x^y:\KM_x\rightarrow \KM_y[1]$$ 
	    the connecting morphism.
	    
	    \begin{prop}\label{prop: connecting morphism holim}
	    Assume that $x\lessdot y$. We have a commutative square in $D(\KMMod)$:
	    $$\xymatrixcolsep{5pc}\xymatrix{
	    \holim_P\KM_x \ar[r]^-{\holim_Pa_x^y}\ar[d]_{\simeq}& \holim_P\KM_y[1]\ar[d]^{\simeq}\\
	    C^{\bullet+1}(x) \ar[r]_{b_x^y[1]}& C^{\bullet+2}(y)
	    }$$
	    where the vertical isomorphisms are those of Proposition \ref{prop:Cx as holim}.
	    \end{prop}
	    
	    \begin{proof}
	    Let $R^\bullet_x$ and $R^\bullet_y$ denote the resolutions of $\KM_x$ and $\KM_y$ described in the proof of Proposition \ref{prop:Cx as holim}. By mimicking the definition of $b_x^y$ and the proof of Lemma \ref{lem: b morphism cohomology} 1) we get a morphism of complexes $R^{\bullet}_x\rightarrow R^{\bullet+1}_y$. We let $S^\bullet$ denote its cone shifted by $-1$, so that $S^\bullet=R^\bullet_x\oplus R^\bullet_y$ as graded $P$-representations. We consider the following commutative diagram where both rows are short exact sequences:
	    $$\xymatrix{
	    0 \ar[r] & \KM_y \ar[r]\ar[d] & \KM_x^y\ar[r]\ar@{-->}[d] & \KM_x \ar[r]\ar[d] & 0 \\
	    0 \ar[r] & R^{\bullet}_y \ar[r] & S^\bullet \ar[r] & R^{\bullet}_x  \ar[r] & 0
	    }$$
	    The dotted arrow $\KM_x^y\rightarrow S^0=\KM_{\leq x}\oplus \KM_{\leq y}$ is defined so that its value at $y$ is the identity of $\KM$ and its value at $x$ is the diagonal morphism $\KM\rightarrow \KM\oplus\KM$. It is a morphism of representations of $P$ because $x\lessdot y$. The composite $\KM^y_x\rightarrow S^0\rightarrow S^1$ is zero, as one can check on the values at $x$ and $y$. In the above commutative diagram, the bottom row is thus a $\lim_P$-acyclic resolution of the top row, by the $5$-lemma. This implies that the connecting morphism $\holim_P\KM_x\rightarrow \holim_P\KM_y[1]$ is computed by the connecting morphism in the long exact sequence associated to the short exact sequence
	    $$0\longrightarrow \lim_PR^\bullet_y\longrightarrow \lim_PS^\bullet \longrightarrow \lim_P R^\bullet_x\rightarrow 0\ .$$
	    By construction, this is nothing but the short exact sequence for the cone of the morphism $b_x^y[-1]:C^{\bullet-1}(x)\rightarrow C^\bullet(y)$, and the connecting morphism is $b_x^y$. The claim follows.
	    \end{proof}
	    
	    \begin{rem}\label{rem: functoriality holim}
	    Let $\alpha:P\to P'$ be an isomorphism of posets, let $x\in P$ and $x'=\alpha(x)\in P'$. One easily proves that the natural isomorphism 
	    $$C^{\bullet+1}_{P'}(x') \simeq \holim_{P'}\KM_{x'} \simeq \holim_P\KM_x \simeq C^{\bullet+1}_P(x)$$
	    is the isomorphism of complexes denoted $C^{\bullet+1}(\alpha)$ in Remark \ref{rem: functoriality complexes C}.
	    \end{rem}

\section{Triangulated derivators}\label{sec: 2}

    In this section we collect some tools about triangulated derivators and natural Postnikov systems arising in this context. The main result is Proposition \ref{prop: pushed}.

    \subsection{The framework of triangulated derivators}
    
        We work within the framework of triangulated derivators, introduced by Grothendieck \cite{grothendieckderivators} and developed by several authors, see \cite{hellerhomotopy, frankeuniqueness, maltsiniotisintroductionderivateurs, ayoubPhD1, cisinskineeman, grothderivatorspointed}. Broadly speaking, triangulated derivators are like triangulated categories with well-defined homotopy limits and colimits (and more generally homotopy Kan extensions).
        
         We work with Ayoub's notion of a triangulated derivator from \cite{ayoubPhD1} in order to be able to use the notion of an algebraic derivator from [\emph{loc. cit.}] in the next section. There a $2$-category of ``diagrams'' is fixed, which is a full sub-$2$-category of the $2$-category of (small) categories satisfying the axioms D0, D1 and D2 in [\emph{loc. cit.}, \S 2.1.2]; we choose it to be the $2$-category of finite posets, since those are the only diagrams that we will need. All $2$-categories in this paper are strict, and our notion of a $2$-functor between two $2$-categories is the weak one, i.e., that of a pseudofunctor in the sense of \cite[7.5]{borceux}. 
        
        \subsubsection{Finite posets}
        
            A finite poset $P$ is viewed as a category with a unique morphism from $x$ to $y$ if $x\leq y$, and none otherwise. Finite posets thus form a full sub-$2$-category of the $2$-category of (small) categories. A functor between finite posets is an order-preserving map and is simply called a morphism of posets. For two such morphisms $f,g:P\to Q$, there is a unique natural transformation from $f$ to $g$ if $f(x)\leq g(x)$ for every $x\in P$, and none otherwise.
            
            We denote by $e$ the poset with one element. For $P$ a finite poset we denote by $p$ or $p_P:P\to e$ the morphism to a point. An inclusion between posets $Q\subset P$ is denoted $i_Q^P:Q\to P$. For $P$ a finite poset and $x\in P$, we use the notation $i_x$ or $i_x^P:e\to P$ for the inclusion of $x$.
        
        \subsubsection{Triangulated pre-derivators}
        
            \begin{defi}
            A \emph{triangulated pre-derivator} $\DM$ is a $1$-contravariant and $2$-covariant $2$-functor from the $2$-category of finite posets to the $2$-category of triangulated categories. In other words, it associates:
            \begin{enumerate}
            \item[0)] to every finite poset $P$, a triangulated category $\DM(P)$;
            \item[1)] to every morphism $f:P\to Q$ of finite posets, a triangulated functor $f^*:\DM(Q)\to \DM(P)$;
            \item[2)] to every pair of morphisms $f,g:P\to Q$ such that $f(x)\leq g(x)$ for every $x\in P$, a natural transformation of triangulated functors $f^*\to g^*$;
            \end{enumerate}
            in a way that is compatible with horizontal and vertical composition.
            \end{defi}
            
            \begin{rem}\label{rem: underlying diagram}
            The triangulated category $\DM(e)$ is called the \emph{ground category}. For a finite poset $P$, an element $x\in P$ and an object $\F\in \DM(P)$, the pullback $(i_x)^*\F\in \DM(e)$ is called the \emph{value} of $\F$ at $x$. For elements $x,y\in P$ such that $x\leq y$ we have two morphisms $i_x,i_y:e\to P$ such that $i_x(\cdot)\leq i_y(\cdot)$ and thus a natural transformation $(i_x)^*\to (i_y)^*$. Thus, the functors $(i_x)^*$ induce an \emph{underlying diagram functor}
            \begin{equation}\label{eq: underlying diagram} \DM(P)\longrightarrow \DM(e)^P\end{equation}
            which is not an equivalence in general. The category $\DM(P)$ should be thought of as the category of ``homotopy coherent'' $P$-shaped diagrams of objects of the ground category $\DM(e)$, whereas the category $\DM(e)^P$ consists of ``homotopy incoherent'' diagrams. More generally we have ``partial'' underlying diagram functors, for finite posets $P$ and $E$,
            $$\DM(E\times P) \longrightarrow \DM(E)^P$$
            and diagrams in $\DM(E)^P$ can be called ``partially homotopy incoherent''.
            \end{rem}
            
            \begin{rem}\label{rem: convention variance derivator}
            Our variance convention slightly differs from that of \cite{ayoubPhD1} since there pre-derivators are $1$-contravariant and $2$-contravariant, which makes the underlying diagram functor land in $\DM(e)^{P^{\mathrm{op}}}$. 
            \end{rem}
            
        \subsubsection{Triangulated derivators}\label{sec: definition derivator}
            
            A \emph{triangulated derivator} \cite[D\'{e}finition 2.1.34]{ayoubPhD1} is a triangulated pre-derivator that satisfies a certain number of axioms, including the following that we mention for future reference.
            
            \begin{enumerate}[1)]
            \item We have $\DM(\varnothing)=0$.
            \item The underlying diagram functor \eqref{eq: underlying diagram} is conservative for every finite poset $P$; it is a triangulated equivalence if $P$ is discrete.
            \item For every morphism $f:P\to Q$ of finite posets, the functor $f^*:\DM(Q)\to \DM(P)$ admits right and left adjoints
            $$f_*:\DM(P)\rightarrow \DM(Q) \hspace{1cm}\mbox{ resp. }\hspace{1cm} f_!:\DM(P)\rightarrow \DM(Q)\ ,$$
            which are automatically triangulated functors. They play the role of homotopy right and left Kan extension functors; in the special case of $p:P\rightarrow e$ the projection to a point, they are a homotopy limit and colimit functors and we write $p_*=\holim$ and $p_!=\hocolim$.
            \item For a morphism $f:P\to Q$ and some element $y\in Q$, let $y/P\subset P$ denote the subposet consisting of elements $x\in P$ such that $y\leq f(x)$. We have a natural transformation 
            $$ (p_{y/P})^*(i_y^Q)^* \to (i_{y/P}^P)^*f^*$$
            associated by $2$-functoriality to the two morphisms $(i_y^Q)\circ p_{y/P}$ and $f\circ (i_{y/P}^P)$ from $y/P$ to $Q$. By using the units and counits of the adjunctions we can obtain from it a natural transformation 
            $$(i_y^Q)^*f_*\longrightarrow (p_{y/P})_*(i_{y/P}^P)^*$$
            which is $(i_y^Q)^*f_*\to (p_{y/P})_*(p_{y/P})^*(i_y^Q)^*f_* \to (p_{y/P})_*(i_{y/P}^P)^*f^*f_* \to (p_{y/P})_*(i_{y/P}^P)^*$.
            We require this last natural transformation to be invertible. In the same vein, let $P/y\subset P$ denote the subposet consisting of elements $x\in P$ such that $f(x)\leq y$. Then we have a natural transformation
            $$(p_{P/y})_!(i_{P/y}^P)^*\longrightarrow (i_y^Q)^*f_!$$
            that we require to be invertible.
            \end{enumerate}
            
            \begin{rem}
            The axioms listed above are similar to the axioms 1.-4. of \cite[D\'{e}finition 2.1.34]{ayoubPhD1}, albeit slightly less complete for 2) and 4). In [\emph{loc. cit.}] two more axioms 5. and 6. relate the triangulated structures on the categories $\DM(P)$ with the homotopy Kan extension functors $f_*$ and $f_!$ and will not be used in the rest of this article.
            \end{rem}
        
            \begin{rem}
            If $\AC$ is a Grothendieck abelian category, e.g., $\AC=\KMMod{}$, then we have a derivator $\DM_\AC$ such that $\DM_\AC(P)$ is the derived category of the diagram category $\AC^P$ for every finite poset $P$. The pullback functors $f^*$ are the obvious ones and their adjoints are obtained by deriving the usual Kan extension functors.
            \end{rem}

        \subsubsection{Monoidal structure} \label{subsubsec: monoidal}
        
            The triangulated derivators that we will deal with all have a unital symmetric monoidal structure in the sense of \cite[\S 2.1.6]{ayoubPhD1}. This means that for every finite poset $P$ the triangulated category $\DM(P)$ is endowed with the structure of a unital symmetric monoidal triangulated category and that for every morphism $f:P\to Q$ the functor $f^*:\DM(Q)\to \DM(P)$ is endowed with the structure of a unital symmetric monoidal functor. The triangulated derivator $\DM_{\KMMod}$ of the abelian category $\KMMod$ is symmetric monoidal.
            
            Let $\DM$ be a unital symmetric monoidal derivator. Then we have, for every morphism of finite posets $f:P\to Q$ and for $\F\in \DM(P)$, $\G\in \DM(Q)$, a natural morphism
            \begin{equation}\label{eq: projection formula star}
            \G \otimes f_*\F \longrightarrow f_*(f^*\G \otimes \F)\ .
            \end{equation}
            It is obtained as the composition
            $\G \otimes f_*\F
            \to f_*f^*(\G \otimes f_*\F)
            \stackrel\sim\to f_*(f^*\G \otimes f^*f_*\F)
            \to f_*(f^*\G \otimes \F)$,
            where the first and last steps involve the unit and the counit of the adjunction, and the middle step uses that $f^*$ is monoidal.
            In the same way, we have a natural morphism
            \begin{equation}\label{eq: projection formula shriek}
            f_!(f^*\G \otimes \F)\longrightarrow \G \otimes f_!\F\ .
            \end{equation}
            Neither \eqref{eq: projection formula star} nor \eqref{eq: projection formula shriek} is an isomorphism in general.
        
        \subsubsection{Coefficients}\label{subsubsec: coefficients}
        
            In the remainder of this section we fix a unital symmetric monoidal triangulated derivator $\DM$ equipped with a morphism of unital symmetric monoidal triangulated derivators $\DM_{\KMMod}\to \DM$.  Such an object can be called a \emph{unital symmetric monoidal triangulated derivator with coefficients in $\KM$}.
            
            We will allow ourselves to interpret complexes of $\KM$-modules as objects of $\DM(e)$ without specific reference to the morphism $\DM_{\KMMod}\to \DM$.

    \subsection{Extension by zero}
    
        We start with a classical lemma.
    
        \begin{lem}\label{lem: greatest least element}
        Let $P$ be a finite poset with projection $p:P\rightarrow e$.
        \begin{enumerate}[1)]
        \item If $P$ has a least element $x$ then we have isomorphisms $p_*\simeq (i_x)^*$ and $p^*\simeq (i_x)_!$. The natural morphism $p_!p^*\rightarrow \mathrm{id}_{\DM(e)}$ is an isomorphism.
        \item If $P$ has a greatest element $y$ then we have isomorphisms $p_!\simeq (i_y)^*$ and $p^*\simeq (i_y)_*$. The natural morphism $\mathrm{id}_{\DM(e)}\rightarrow p_*p^*$ is an isomorphism.
        \end{enumerate}
        \end{lem}
        
        \begin{proof}
        We prove the first point (the second is proved dually). The fact that $x$ is the least element of $P$ may be expressed by the fact that $(i_x,p)$ is an adjoint pair of functors. It follows that $(p^*,(i_x)^*)$ is also an adjoint pair of functors. Now $(i_x)^*$ being a right adjoint to $p^*$ means that it is equal to $p_*$, and $p^*$ being a left adjoint to $(i_x)^*$ means that it is equal to $(i_x)_!$.
        
        For the second assertion, note that $p i_x = \id_e$, hence $p_! (i_x)_! \simeq \id_{\DM(e)}$, and the isomorphism $p^* = (i_x)_!$ identifies this with the adjunction morphism $p_! p^* \to \id_{\DM(e)}$.
        \end{proof}
        
        \begin{lem}\label{lem: extension subposet}
        Let $i:Q\hookrightarrow P$ denote the inclusion of a subposet. For every $\mathcal{G}\in \DM(Q)$ the natural morphisms
        $$i^*i_*\G\longrightarrow \G \qquad \mbox{and} \qquad \G\longrightarrow i^*i_!\G$$
        are isomorphisms.
        \end{lem}
        
        \begin{proof}
        We prove that the first morphism is an isomorphism (the second case is proved dually). For every $x\in Q$ we have a sequence of isomorphisms
        $$(i_x^Q)^*i^*i_*\G\simeq (i_x^P)^*i_*\G\simeq (p_{x/Q})_*(i_{x/Q}^Q)^*\G\simeq (i_x^{x/Q})^*(i_{x/Q}^Q)^*\G\simeq (i_x^Q)^*\G$$
        where the second isomorphism follows from \S\ref{sec: definition derivator} 4) and the third isomorphism follows from Lemma \ref{lem: greatest least element} since $x$ is the least element of $x/Q$.
        One checks that the composition of these isomorphisms is the composition of $(i_x^Q)^*$ with the natural morphism $i^*i_*\G\to \G$. By \S\ref{sec: definition derivator} 2) this proves the claim.
        \end{proof}
    
        \begin{defi}
        Let $P$ be a finite poset.
        \begin{enumerate}[1)]
        \item A \emph{sieve} in $P$ is a subset $U\subset P$ such that for every $x\leq y$ in $P$ we have $y\in U\Rightarrow x\in U$. 
        \item A \emph{cosieve} in $P$ is a subset $V\subset P$ such that for every $x\leq y$ in $P$ we have $x\in V\Rightarrow y\in V$. 
        \end{enumerate}
        \end{defi}
        
        The complement of a sieve is a cosieve and the complement of a cosieve is a sieve. We also call a sieve (resp. cosieve) the functor of posets given by the inclusion of a sieve (resp. cosieve). The following lemma is classical and says that the functor $u_*$ (resp. $v_!$) deserves the name ``extension by zero'' if $u$ is a sieve (resp. if $v$ is a cosieve).
        
        \begin{lem}\label{lem: extension by zero}
        Let $P$ be a finite poset.
        \begin{enumerate}[1)]
        \item Let $u: U \hookrightarrow P$ be a sieve. For $\mathcal{F}\in\DM(P)$, the natural morphism $\mathcal{F}\rightarrow u_*u^*\mathcal{F}$ is an isomorphism if and only if $(i_x)^*\mathcal{F}=0$ for all $x \in P\setminus U$.
        \item Let $v: V \hookrightarrow P$ be a cosieve. For $\mathcal{F}\in\DM(P)$, the natural morphism $v_!v^*\mathcal{F}\rightarrow \mathcal{F}$ is an isomorphism if and only if $(i_x)^*\mathcal{F}=0$ for all $x \in P\setminus V$.
        \end{enumerate}
        \end{lem}
        
        \begin{proof}
        We prove the first point (the second is proved dually). Let us assume that the natural morphism $\F\to u_*u^*\F$ is an isomorphism. Then for $x\in P\setminus U$ we have an isomorphism
        $$(i_x)^*\F\simeq (i_x)^*u_*u^*\F\simeq (p_{x/U})_*(i_{x/U}^U)^*u^*\F$$
        where the second isomorphism follows from \S\ref{sec: definition derivator} 4). By assumption we have $x/U=\varnothing$ and \S\ref{sec: definition derivator} 1) implies that we have $(i_x)^*\F=0$. Conversely, if $(i_x)^*\F=0$ for all $x\in P\setminus U$ then the same argument shows that the natural morphism 
        $(i_x)^*\F\to (i_x)^*u_*u^*\F$ is an isomorphism. The fact that it is an isomorphism also for $x\in U$ follows from the same kind of reasoning as in the proof of Lemma \ref{lem: extension subposet}. Thanks to \S\ref{sec: definition derivator} 2) we conclude that the morphism $\F\to u_*u^*\F$ is an isomorphism.
        \end{proof}

        
        
        
        The next lemma explains the compatibility between extension by zero and pullback.
        
        \begin{lem}\label{lem: base change extension by zero}
        \begin{enumerate}[1)]
        \item Let us consider the following cartesian diagram in the category of finite posets, where $u$ is a sieve.
        $$\xymatrix{
        f^{-1}(U) \ar[r]^-{u'}\ar[d]_{g}& Q\ar[d]^{f} \\
        U \ar[r]_-{u}& P
        }$$
        Then we have a canonical isomorphism $f^*u_*\stackrel{\sim}{\longrightarrow} (u')_*g^*$.
        \item Let us consider the following cartesian diagram in the category of posets, where $v$ is a cosieve.
        $$\xymatrix{
        f^{-1}(V) \ar[r]^-{v'}\ar[d]_{h}& Q\ar[d]^{f} \\
        V \ar[r]_-{v}& P
        }$$
        Then we have a canonical isomorphism $(v')_!h^*\stackrel{\sim}{\longrightarrow} f^*v_!$.
        \end{enumerate}
        \end{lem}
        
        \begin{proof}
        We prove the first point (the second is proved dually). The morphism $f^*u_*\to (u')_*g^*$ is the composite $f^*u_*\to (u')_*(u')^*f^*u_*\stackrel{\sim}{\to} (u')_*g^*u^*u_* \stackrel{\sim}{\to}(u')_*g^*$. The fact that it is an isomorphism follows from Lemma \ref{lem: extension by zero} and the fact that $u$ and $u'$ are sieves.
        \end{proof}
        
        The next lemma provides a projection formula for the ``extension by zero'' functors.
        
        \begin{lem}\label{lem: projection formula sieve cosieve}
        Let $P$ be a finite poset.
        \begin{enumerate}[1)]
        \item Let $u:U\hookrightarrow P$ be a sieve. For $\mathcal{F}\in \DM(P)$ and $\mathcal{G}\in\DM(U)$ the natural morphism
        $$\mathcal{F}\otimes u_*\mathcal{G}
        \longrightarrow
        u_*(u^*\mathcal{F}\otimes\mathcal{G})$$
        defined in \S \ref{subsubsec: monoidal} \eqref{eq: projection formula star} is an isomorphism.
        \item Let $v:V\hookrightarrow P$ be a cosieve. For $\mathcal{F}\in \DM(P)$ and $\mathcal{G}\in\DM(V)$ the a natural morphism
        $$v_!(v^*\mathcal{F}\otimes\mathcal{G})
        \longrightarrow
        \mathcal{F}\otimes v_!\mathcal{G}$$
        defined in \S \ref{subsubsec: monoidal} \eqref{eq: projection formula shriek} is an isomorphism.
        \end{enumerate}
        \end{lem}
        
        \begin{proof}
        We prove the first point (the second is proved dually). Let $c:P\setminus U\hookrightarrow P$ denote the cosieve complementary to $u$. Then we have $c^*(\F\otimes u_*\G)\simeq c^*\F\otimes c^*u_*\G=0$ since $c^*u_*=0$ by Lemma \ref{lem: extension by zero}. Using that same lemma and also Lemma \ref{lem: extension subposet}, we see that each step in the definition of the morphism \S \ref{subsubsec: monoidal} \eqref{eq: projection formula star} is an isomorphism.
        \end{proof}

    \subsection{Localization triangles}
    
        Let $P$ be a finite poset. Let $u:U\hookrightarrow P$ be a sieve and $v:V\hookrightarrow P$ denote the complementary cosieve.
    
        \begin{lem}\label{lem:localizationtrianglederivator}
        For $\mathcal{F}\in \DM(P)$ there is a unique distinguished triangle in $\DM(P)$ 
        \begin{equation}\label{eq: localization triangle derivator}
        v_!v^*\mathcal{F} \longrightarrow \mathcal{F} \longrightarrow u_*u^*\mathcal{F} \stackrel{+1}{\longrightarrow}\ 
        \end{equation}
        such that the first two maps are the counit and unit respectively. It is functorial in $\mathcal{F}$. We call it a \emph{localization triangle}.
        \end{lem}
        
        \begin{proof}
        Let $C$ denote a cone of the counit morphism $v_!v^*\mathcal{F}\rightarrow \mathcal{F}$, so that we have a distinguished triangle in $\DM(P)$:
        \begin{equation}\label{eq:coneofcounit}
        v_!v^*\mathcal{F} \longrightarrow \mathcal{F} \longrightarrow C \stackrel{+1}{\longrightarrow}\ .
        \end{equation}
        By applying the triangulated functor $v^*$ to \eqref{eq:coneofcounit} and using Lemma \ref{lem: extension subposet} we get a distinguished triangle in $\DM(V)$:
        $$v^*\mathcal{F}\stackrel{\mathrm{id}}{\longrightarrow} v^*\mathcal{F}\longrightarrow v^*C \stackrel{+1}{\longrightarrow}$$
        We thus have $v^*C=0$ and Lemma \ref{lem: extension by zero} implies that we have an isomorphism $C\simeq u_*u^*C$. By applying the triangulated functor $u^*$ to \eqref{eq:coneofcounit} and using $u^*v_!=0$, which follows from Lemma \ref{lem: extension by zero}, we get a distinguished triangle in $\DM(U)$:
        $$0\longrightarrow u^*\mathcal{F} \longrightarrow u^*C\stackrel{+1}{\longrightarrow}$$
        and deduce that we have an isomorphism $C\simeq u_*u^*\mathcal{F}$. This implies the existence of a distinguished triangle whose first two edges are the counit $v_!v^*\mathcal{F}\rightarrow \mathcal{F}$ and the unit $\mathcal{F}\rightarrow u_*u^*\mathcal{F}$. By adjunction and $v^*u_*=0$, which follows from Lemma \ref{lem: extension by zero}, we have $\mathrm{Hom}_{\DM(P)}(v_!v^*\mathcal{F},u_*u^*\mathcal{F}[-1])=0$, and \cite[Corollaire 1.1.10]{bbd} implies that the remaining edge of the triangle is unique. This implies that the triangle is functorial in $\mathcal{F}$.
        \end{proof}
        
        \begin{rem}\label{rem: totally coherent}
        The output of the above lemma, as well as the results of the rest of this section, is a diagram in the triangulated category $\DM(P)$, and is thus a partially incoherent diagram from the point of view of derivators (see Remark \ref{rem: underlying diagram}). It is of course possible to lift it to a coherent diagram living in $\DM(P\times [3])$, where $[n]$ denotes the poset $(\{0,1,\ldots,n\},\leq)$ with $n$ consecutive arrows. We choose not to phrase our results (and in particular, Proposition \ref{prop: pushed} below) in this totally coherent way but rather in a way that is more appealing to readers familiar with the setting of triangulated categories.
        
        However, let us sketch a way to do so in the particular example of the above lemma. The first step is to lift the counit morphism $v_!v^*\F\to \F$ to an object of $\DM(P\times [1])$. For this we can consider the cosieve $v':V'\hookrightarrow P\times [1]$ where $V'$ consists of those elements $(x,i)$ such that $x\in V$ if $i=0$. If $f:P\times [1]\to P$ denotes the natural projection, then we can consider the object
        $$(v')_!(v')^*f^*\F \in \DM(P\times[1])$$
        and check that its underlying morphism in $\DM(P)$ is indeed the counit morphism $v_!v^*\F\to \F$. One can then proceed as in \cite[\S 4.2]{grothderivatorspointed} (see also \cite[Remarque 2.1.38]{ayoubPhD1}) to produce a coherent lift of the triangle \eqref{eq:coneofcounit}, and the same arguments as in the proof above identify it to a coherent lift of the triangle \eqref{eq: localization triangle derivator}.
        \end{rem}
        
        
        
        
        The next lemma explains the compatibility between the localization triangles and pullback.
        
        \begin{lem}\label{lem: base change localization triangle}
        Let $f:Q\rightarrow P$ be a morphism of finite posets and introduce a sieve $u':f^{-1}(U)\hookrightarrow Q$ and a cosieve $v':f^{-1}(V)\hookrightarrow Q$. For $\F\in \DM(P)$ we have the following isomorphism of distinguished triangles, where the first triangle is obtained by applying $f^*$ to \eqref{eq: localization triangle derivator} and the second triangle is the localization triangle \eqref{eq: localization triangle derivator} of $f^*\F$ with respect to $u'$ and $v'$.
        $$\xymatrix{
        f^*v_!v^*\F \ar[r] & f^*\F \ar@{=}[d] \ar[r] & f^*u_*u^*\F \ar[r]^-{+1} \ar[d]_{\simeq}& \\
        (v')_!(v')^*f^*\F  \ar[u]_{\simeq} \ar[r] & f^*\F \ar[r]& (u')_*(u')^*f^*\F \ar[r]^-{+1}  &
        }$$
        \end{lem}
        
        \begin{proof}
        It is obtained from the following diagram, where the notation is borrowed from Lemma \ref{lem: base change extension by zero}.
        $$\xymatrix{
        f^*v_!v^*\F \ar[r] & f^*\F \ar@{=}[dd] \ar[r] & f^*u_*u^*\F \ar[r]^-{+1} \ar[d]_{\simeq}& \\
        (v')_!h^*v^*\F \ar@{<->}[d]_{\simeq} \ar[u]^{\simeq} && (u')_*g^*u^*\F & \\
        (v')_!(v')^*f^*\F  \ar[r] & f^*\F \ar[r]& (u')_*(u')^*f^*\F \ar[r]^-{+1} \ar@{<->}[u]_{\simeq} &
        }$$
        The isomorphisms between the first and second row follow from Lemma \ref{lem: base change extension by zero}. The two visible squares of the diagram commute, and the remaining square commutes by the uniqueness statement in Lemma \ref{lem:localizationtrianglederivator}.
        \end{proof}
        
        \begin{lem}\label{lem: projection formula localization triangle}
        For $\F,\F'\in \DM(P)$ we have the following isomorphism of distinguished triangles, where the rows are (induced by) localization triangles:
        $$\xymatrix{
        v_!v^*(\F\otimes\F') \ar[r]\ar[d]_{\simeq} & \F\otimes \F' \ar@{=}[d] \ar[r] & u_*u^*(\F\otimes \F') \ar[r]^-{+1}& \\
        \F\otimes v_!v^*\F'  \ar[r] & \F\otimes \F' \ar[r]& \F\otimes u_*u^*\F'\ar[r]^-{+1} \ar[u]_{\simeq} &
        }$$
        \end{lem}
        
        \begin{proof}
        It is obtained from the diagram:
        $$\xymatrix{
        v_!v^*(\F\otimes\F') \ar[r]\ar@{<->}[d]_{\simeq} & \F\otimes \F' \ar@{=}[dd] \ar[r] & u_*u^*(\F\otimes \F') \ar@{<->}[d]^{\simeq}\ar[r]^-{+1}& \\
        v_!(v^*\F\otimes v^*\F') \ar[d]_{\simeq} && u_*(u^*\F\otimes u^*\F') & \\
        \F\otimes v_!v^*\F'  \ar[r] & \F\otimes \F' \ar[r]& \F\otimes u_*u^*\F'\ar[r]^-{+1} \ar[u]_{\simeq} &
        }$$
        The isomorphisms between the second and third row follow from Lemma \ref{lem: projection formula sieve cosieve}. The two visible squares of the diagram commute, and the remaining square commutes by the uniqueness statement in Lemma \ref{lem:localizationtrianglederivator}.
        \end{proof}
        
        For $x<y$ in $P$ and $\F\in \DM(P)$ let us denote by $(i_{x<y})^*\F:(i_x)^*\F\to (i_y)^*\F$ the corresponding morphism in $\DM(e)$ in the underlying diagram (see Remark \ref{rem: underlying diagram}). Recall from \S\ref{subsec: poset cohomology holim} the morphism $a_x^y:\KM_x\to \KM_y[1]$ in $\DM_{\KMMod}(P)$.
        
        \begin{lem}\label{lem: connecting morphism discrete posets}
        Assume that $U$ and $V$ are discrete posets. Then the connecting morphism in the localization triangle \eqref{eq: localization triangle derivator} reads
        $$u_*u^*\F \simeq \bigoplus_{x\in U} p^*(i_x)^*\F\otimes \KM_x \longrightarrow  \bigoplus_{y\in V} p^*(i_y)^*\F\otimes \KM_y[1] \simeq v_!v^*\F[1]$$
        where the component indexed by $x\in U$ and $y\in V$ is $p^*(i_{x<y})^*\F\otimes a_x^y$ if $x<y$ and zero otherwise.
        \end{lem}
        
        Note that the object $p^*(i_x)^*\F\otimes \KM_x \in \DM(P)$ has value $(i_x)^*\F$ at $x$ and zero at every other point.
        
        \begin{proof}
        We proceed in two steps.
        \begin{enumerate}[1)]
        \item Assume that we work in the derivator $\DM_{\KMMod}$ and that $\F=p^*\KM\in\DM(P)$ is the constant object with values $\KM$. Since $U$ and $V$ are discrete posets we have, by \S\ref{sec: definition derivator} 2), isomorphisms
        $$u_*u^*p^*\KM\simeq \bigoplus_{x\in U}\KM_x \qquad \mbox{and}\qquad v_!v^*p^*\KM \simeq \bigoplus_{y\in V}\KM_y\ .$$
        For $x\in U$ and $y\in V$, we can apply Lemma \ref{lem: base change localization triangle} to $Z=\{x,y\}$, to reduce the computation of the connecting morphism to the case where $P=Z$ has two elements. If $x<y$ then the connecting morphism is $a_x^y$ by definition. Otherwise $P$ is itself discrete and \S\ref{sec: definition derivator} 2) implies that we have $\F\simeq u_*u^*p^*\KM\oplus v_!v^*p^*\KM$, and the connecting morphism is zero. The claim follows.
        \item We now work in the general case of the lemma. We write $\F=\F\otimes p^*\KM$. By applying Lemma \ref{lem: projection formula localization triangle} for $\F'=p^*\KM$ and using the first step of the proof, we get a commutative diagram
        $$\xymatrix{
        u_*u^*\F \ar[r] & v_!v^*\F[1]\ar[d]^{\simeq} \\
        \F\otimes u_*u^*p^*\KM \ar[u]^{\simeq}\ar@{<->}[d]_{\simeq} \ar[r] & \F\otimes v_!v^*p^*\KM[1] \ar@{<->}[d]^{\simeq} \\
        \displaystyle\bigoplus_{x\in U}\F\otimes \KM_x \ar[r]& \displaystyle\bigoplus_{y\in V}\F\otimes \KM_y[1]
        }$$
        where the component of the bottom morphism indexed by $x\in U$ and $y\in V$ is $\mathrm{id}_\F\otimes a_x^y$ if $x<y$ and zero otherwise. Let us now fix $x\in U$ and $y\in V$ with $x<y$. By $2$-functoriality we have a commutative diagram
        $$\xymatrix{
        \F \ar@{=}[r] & \F \ar[d]\\
        (i_x)_!(i_x)^*\F \ar[u]& (i_y)_*(i_y)^*\F \\
        (i_x)_!(i_x)^*p^*(i_x)^*\F \ar@{=}[u]\ar[d] & (i_y)_*(i_y)^*p^*(i_y)^*\F \ar@{=}[u] \\
        p^*(i_x)^*\F \ar[r]_{p^*(i_{x<y})^*\F}& p^*(i_y)^*\F\ar[u]
        }$$
        where the values at $x$ of the vertical arrows on the left are isomorphisms and the values at $y$ of the vertical arrows on the right are isomorphisms. We then conclude that we have a commutative diagram
        $$\xymatrixcolsep{6pc}\xymatrix{
        \F\otimes \KM_x \ar[r]^-{\mathrm{id}\otimes a_x^y} \ar@{<->}[d]_{\simeq}& \F\otimes \KM_y[1]\\
        p^*(i_x)^*\F\otimes \KM_x \ar[r]_{p^*(i_{x<y})^*\F\otimes a_x^y}& p^*(i_y)^*\F\otimes \KM_y[1]\ar@{<->}[u]_{\simeq}
        }$$
        and the claim follows.
        \end{enumerate}
        \end{proof}
        
    \subsection{Postnikov systems from derivators}   
    
        Let $P$ be a finite poset and let $\sigma:P\to\mathbb{Z}_{\geq 1}$ be a strictly increasing map. This defines a finite decreasing filtration of $P$ by cosieves $V^k=\{x\in P \; ,\; \sigma(x)\geq k\}$ such that each complement $V^k\setminus V^{k+1}$ is a discrete poset (an antichain in $P$). We let $v^k:V^k\hookrightarrow P$.
        
        \begin{lem}\label{lem:postnikovsystemderivator}
        Let $\mathcal{F}\in\DM(P)$.
        \begin{enumerate}[1)]
        \item We set $F^k\mathcal{F}=(v^k)_!(v^k)^*\mathcal{F}$. We have a Postnikov system in $\DM(P)$:
        $$\xymatrix{
		\cdots &\ar[rr]&& F^3\mathcal{F} \ar[rr]\ar[ld] && F^2\mathcal{F}\ar[rr]\ar[ld]  && F^1\mathcal{F}=\mathcal{F} \ar[ld]\\
		&&G^3\mathcal{F} \ar[lu]^{+1}&& G^2\mathcal{F}\ar[ul]^{+1} && G^1\mathcal{F} \ar[ul]^{+1}& 
		}$$
		where the graded objects are given by 
		$$G^k\F\simeq \bigoplus_{\sigma(x)=k}p^*(i_x)^*\F\otimes \KM_x\ .$$
		\item For every integer $k$, the connecting morphism $G^k\F\rightarrow G^{k+1}\F[1]$ has its component indexed by $x$ and $y$ with $\sigma(x)=k$, $\sigma(y)=k+1$, given by
		$$\xymatrixcolsep{5pc}\xymatrix{ p^*(i_x)^*\F\otimes \KM_x \ar[r]^-{p^*(i_{x<y})^*\F\otimes a_x^y} & p^*(i_y)^*\F\otimes \KM_y[1] }$$
		if $x<y$, and zero otherwise.
		\item The above Postnikov system is functorial in $\F$.
		\end{enumerate}
        \end{lem}
        
        \begin{proof}
        \begin{enumerate}[1)]
        \item The morphism $F^{k+1}\F\to F^k\F$ is defined as the composite
        $$(v^{k+1})_!(v^{k+1})^*\F\simeq (v^k)_!v_!v^*(v^k)^*\F\longrightarrow (v^k)_!(v^k)^*\F$$
        where $v:V^{k+1}\hookrightarrow V^{k}$ is a cosieve with complementary sieve $u:V^{k}\setminus V^{k+1}\hookrightarrow V^k$.
        According to Lemma \ref{lem:localizationtrianglederivator} this morphism fits into a distinguished triangle
        $$F^{k+1}\F\longrightarrow F^k\F \longrightarrow G^k\F\stackrel{+1}{\longrightarrow}$$
        with $G^k\F=(v^k)_!u_*u^*(v^k)^*\F$. Since $V^k\setminus V^{k+1}$ is a discrete poset we have as in the proof of Lemma \ref{lem: connecting morphism discrete posets} an isomorphism
        $$G^k\F\simeq \bigoplus_{\sigma(x)=k}\F\otimes \KM_x \simeq \bigoplus_{\sigma(x)=k} p^*(i_x)^*\F\otimes \KM_x\ .$$
        \item Applying Lemma \ref{lem: base change localization triangle} to $Z=\{x\in P ,\, \sigma(x)\in \{k,k+1\}\}$ we are reduced to the two-step case where $\sigma(P)\subset \{1,2\}$. In this case the claim is Lemma \ref{lem: connecting morphism discrete posets} and we are done.
        \item The functoriality statement follows from the functoriality of localization triangles (Lemma \ref{lem:localizationtrianglederivator}).
        \end{enumerate}
        \end{proof}
        
        \begin{rem}
        In the spirit of Remark \ref{rem: totally coherent} let us sketch a way to lift the partially incoherent Postnikov system of the above lemma to a totally coherent diagram\footnote{This was suggested to us by Martin Gallauer.}. The first step is to lift the horizontal morphisms to an object of $\DM(P\times [n])$ where $n$ is an integer such that $\sigma(P)\subset \{1,\ldots,n\}$. For this we consider the cosieve $v':V'\hookrightarrow P\times [n]$ consisting of elements $(x,i)$ such that $x\in V^{i+1}$. If $f:P\times [n]\to P$ denotes the natural projection then the object $(v')_!(v')^*f^*\F\in \DM(P\times [n])$ is a coherent lift of the composable morphisms $F^{k+1}\F\to F^{k}\F$ in $\DM(P)$. One can then produce the remainder of the Postnikov system in a coherent way as in Remark \ref{rem: totally coherent}.
        \end{rem}
        
        In the next section we will apply the functor $p_*$ to a Postnikov system as in Lemma \ref{lem:postnikovsystemderivator}. For this reason we now recast poset cohomology in the context of a general monoidal triangulated derivator.
        
        \begin{lem}\label{lem: projection formula}
        Let $P$ be a finite poset and let $x\in P$. For $M\in \DM(e)$ we have a functorial isomorphism:
        $$p_*(p^*M\otimes \KM_x)\simeq M\otimes C^{\bullet+1}(x)\ .$$
        \end{lem}
        
        \begin{proof}
        Call $\mathcal{F}\in \DM(P)$ \emph{admissible} if for any $M\in \DM(e)$, the natural morphism 
        \begin{equation}\label{eq: projection proof}
        M\otimes p_*\mathcal{F} \longrightarrow p_*(p^*M\otimes\mathcal{F})
        \end{equation}
        defined in \S \ref{subsubsec: monoidal} \eqref{eq: projection formula star} is an isomorphism. Admissible objects satisfy the following properties.
        \begin{enumerate}[(a)]
        \item If $P$ has a greatest element then for every $N\in \DM(e)$, $p^*N$ is admissible. Indeed by Lemma \ref{lem: greatest least element} we have $p_*p^*\simeq \mathrm{id}_{\DM(e)}$ and \eqref{eq: projection proof} is isomorphic to the identity of $M\otimes N$.
        \item If $u:U\hookrightarrow P$ is a sieve and $\mathcal{G}\in \DM(U)$ is admissible, then $u_*\mathcal{G}$ is admissible. Indeed, let $v:P\setminus U\hookrightarrow P$ denote the cosieve complementary to $U$, we have $v^*(p^*M\otimes u_*\mathcal{G})\simeq v^*p^*M\otimes v^*u_*\mathcal{G}=0$ since $v^*u_*=0$. By Lemma \ref{lem: extension by zero} we thus have an isomorphism $p^*M\otimes u_*\mathcal{G} \simeq u_*u^*(p^*M\otimes u_*\mathcal{G})\simeq u_*((p\circ u)^*M\otimes \mathcal{G})$, and \eqref{eq: projection proof} is isomorphic to the natural morphism
        $$M\otimes (p\circ u)_*\mathcal{G} \longrightarrow (p\circ u)_*((p\circ u)^*M\otimes \mathcal{G})\ ,$$
        which is an isomorphism because $\mathcal{G}$ is admissible by assumption.
        \item By the naturality of \eqref{eq: projection proof}, an extension of admissible objects (and in particular a finite direct sum of admissible objects) is admissible. A shift of an admissible object is admissible.
        \end{enumerate}
        We now note that we have, as in the proof of Proposition \ref{prop:Cx as holim}, a resolution $\KM_x\stackrel{\sim}{\rightarrow} R^\bullet_x$ with 
        $$R^n_x = \bigoplus_{[x_1<\cdots <x_n<x_{n+1}=x]} \KM_{\leq x_1}\ .$$
        For every $y\leq x$ we have $\KM_{\leq y} \simeq (u_{\leq y})_*(p_{\leq y})^*\KM$, where $u_{\leq y}:P_{\leq y}\hookrightarrow P$ and $p_{\leq y}:P_{\leq y}\rightarrow e$ are the inclusion and projection maps of the subposet $P_{\leq y}=\{a\in P\, , \, a\leq y\}$. Since $y$ is the greatest element of $P_{\leq y}$, we get by (a) above that $(p_{\leq y})^*\KM$ is admissible. Since $P_{\leq y}$ is a sieve in $P$, we get by (b) above that $\KM_{\leq y}$ is admissible. By (c) above we thus get that every $R^n_x$ is admissible and then that $\KM_x$ is admissible. The claim then follows from Proposition \ref{prop:Cx as holim} since $p_*$ is the homotopy limit functor.
        \end{proof}
        
        The next proposition will be our main tool in the next section. It computes a homotopy limit in the shape of a Postnikov system.
        
        \begin{prop}\label{prop: pushed}
        Let $\mathcal{F}\in\DM(P)$.
        \begin{enumerate}[1)]
        \item We set $F^kp_*\mathcal{F}=p_*(v^k)_!(v^k)^*\mathcal{F}$. We have a functorial Postnikov system in $\DM(e)$:
        $$\xymatrix{
		\cdots & \ar[rr] && F^2p_*\mathcal{F}\ar[rr]\ar[ld]  && F^1p_*\mathcal{F}=p_*\mathcal{F} \ar[ld]\\
		&& G^2p_*\mathcal{F}\ar[ul]^{+1} && G^1p_*\mathcal{F} \ar[ul]^{+1}& 
		}$$
		where the graded objects are given by 
		$$G^kp_*\F\simeq \bigoplus_{\sigma(x)=k}(i_x)^*\F\otimes C^{\bullet+1}(x)\ .$$
		\item For every integer $k$, the connecting morphism $G^kp_*\F\rightarrow G^{k+1}p_*\F[1]$ has its component indexed by $x$ and $y$ with $\sigma(x)=k$, $\sigma(y)=k+1$, given by
		$$\xymatrixcolsep{5pc}\xymatrix{ (i_x)^*\F\otimes C^{\bullet+1}(x) \ar[r]^-{(i_{x<y})^*\F\otimes b_x^y[1]} & (i_y)^*\F\otimes C^{\bullet+2}(y)}$$
		if $x<y$, and zero otherwise.
		\item The above Postnikov system is functorial in $\F$.
		\end{enumerate}
        \end{prop}
        
        \begin{proof}
        This follows from applying the triangulated functor $p_*$ to the Postnikov system of Lemma \ref{lem:postnikovsystemderivator} and setting $F^kp_*\F:=p_*F^k\F$ and $G^kp_*\F:=p_*G^k\F$. The description of the graded objects follows from Lemma \ref{lem: projection formula}. The description of the connecting morphisms follows from Proposition \ref{prop: connecting morphism holim}.
        \end{proof}
        
        \begin{rem}\label{rem: functoriality combinatorial postnikov system}
        The Postnikov system of Proposition \ref{prop: pushed} is functorial with respect to isomorphisms of posets in the following sense. Let $\alpha:P\to P'$ be an isomorphism of posets; we set $\sigma'=\sigma\circ\alpha^{-1}$. For $\F'\in \DM(P')$ there is a natural isomorphism $(p')_*\F'\stackrel{\sim}{\to} p_*\alpha^*\F'$ and a natural isomorphism between the Postnikov system corresponding to $\F'\in \DM(P')$ and the one corresponding to $\alpha^*\F'\in \DM(P)$. The corresponding isomorphism at the level of graded objects has component indexed by $x'\in P'$ and $x\in P$ given by 
        $$\xymatrixcolsep{5pc}\xymatrix{ (i_{x'})^*\F'\otimes C_{P'}^{\bullet+1}(x') \ar[r]^{\id\otimes C^{\bullet+1}(\alpha)}_{\sim} & (i_x)^*\alpha^*\F'\otimes C^{\bullet+1}_P(x) }$$
        if $\alpha(x)=x'$ and zero otherwise, where $C^{\bullet+1}(\alpha)$ was defined in Remark \ref{rem: functoriality complexes C}. This follows easily from Remark \ref{rem: functoriality holim}.
        \end{rem}

\section{The main theorem}\label{sec: sec main theorem}

    \subsection{Categories of motives} 
    
        \subsubsection{Conventions on schemes}
        
            In what follows we fix a noetherian base scheme $B$ and write ``scheme'' for ``separated scheme over $B$''.
            
        \subsubsection{Motives over a scheme}
    
            For every scheme $X$ we have, following Morel--Voevodsky \cite{morelvoevodsky} and Ayoub \cite{ayoubPhD1, ayoubPhD2}, a unital symmetric monoidal triangulated derivator $\DA_X$ of \'{e}tale motives over $X$ with coefficients in $\KM$. 
            It is a particular case of a stable homotopical functor $\mathbb{SH}_{\mathfrak{M}}^T$ constructed in
            \cite[D\'{e}finition 4.5.21]{ayoubPhD2}, taking for the model category $\mathfrak{M}$ (the category of ``coefficients'') the category of complexes of $\KM$-modules, for $T$ the Tate motive (the stabilization consists in formally inverting the functor $T\otimes-$), and considering the \'{e}tale topology;
            the axioms of a unital symmetric monoidal triangulated derivator are proved to hold
            in [\emph{loc. cit.}, \S 4.5]. Other constructions lead to equivalent (under certain assumptions) categories of motives, such as Beilinson motives, \'{e}tale motives with transfers, and $h$-motives (see \cite[\S 16.2]{cisinskideglise}, \cite[Th\'{e}or\`{e}me B.1]{ayoubetale} and \cite[Corollary 5.5.5]{cisinskidegliseetale}). 
            
            \begin{rem}
            By making other choices of $\mathfrak{M}$ and $T$ one is led to other categories such as the Morel--Voevodsky stable $\mathbb{A}^1$-homotopy categories of schemes $\mathbb{SH}$, where our results below still hold.
            \end{rem}
            
            There is a natural morphism of unital symmetric monoidal triangulated derivators $\DM_{\KMMod}\to \DA_X$, so that the derivator $\DM=\DA_X$ satisfies the assumptions of \S\ref{subsubsec: coefficients}. In what follows we will make an abuse of notation and simply write $\DA_X$ for the ground category $\DA_X(e)$.
            
            Let us note that $X\mapsto \DA_X$ satisfies the ``six functor formalism'', for which we will not give a definition here but rather refer to Ayoub. This means that it has the same formal functoriality properties as derived categories of sheaves in familiar contexts. In particular, it underlies a cross functor \cite[D\'{e}finition 1.2.12, Scholie 1.4.2]{ayoubPhD1}. This notion (defined in \S 1.2 in \emph{loc. cit.}) abstracts the properties of the exchange morphisms between $!$ and $*$ pullbacks and/or pushforwards (such as the morphism appearing in the proper base change theorem).
            
            Another important feature that we will use is the existence of functorial localization triangles [\emph{loc. cit.}, \S 1.4.4] for $\F\in\DA_X$, where $i:Z\hookrightarrow X$ denotes a closed immersion and $j:X\setminus Z\hookrightarrow X$ denotes the complementary open immersion:
            \begin{equation}\label{eq: localization triangle geometric}
            j_!j^!\F\longrightarrow \F\longrightarrow i_*i^*\F\stackrel{+1}{\longrightarrow}\ .
            \end{equation}

        \subsubsection{Motives over a diagram of schemes}
        
            In the proof of the main theorem below we will make use of categories of motives over diagrams of schemes, introduced by Ayoub. A \emph{diagram of schemes} $(P,\mathcal{X})$ is the datum of a finite poset $P$ along with a functor $\mathcal{X}:P^{\mathrm{op}}\to \Sch$. (Our convention is actually opposed to Ayoub's, see Remark \ref{rem: conventions diagrams schemes} below.) For $X$ a scheme we have the constant diagram of schemes $(P,X)$ where all the transition maps are the identity of $X$. We view a scheme as the constant diagram of schemes on the poset with one element: $X=(e,X)$. Diagrams of schemes form a $2$-category \cite[D\'{e}finition 2.4.4]{ayoubPhD1} in which a morphism $\alpha:(P,\mathcal{X})\to (Q,\mathcal{Y})$ consists of a morphism of posets $\alpha:P\to Q$ along with a natural transformation $\mathcal{X}\Rightarrow \mathcal{Y}\circ \alpha$. 
        
            Ayoub defines a ($1$-contravariant, $2$-covariant) $2$-functor 
            $$(P,\mathcal{X})\mapsto \DA(P,\mathcal{X})$$ 
            from the $2$-category of diagrams of schemes to the $2$-category of triangulated categories which extends the derivator $P\mapsto \DA(P,X)=\DA_X(P)$ for every scheme $X$. This functor satisfies the axioms of an \emph{algebraic derivator} \cite[2.4.2]{ayoubPhD1} that we will not discuss here. We simply note that for $\alpha:(P,\mathcal{X})\to (Q,\mathcal{Y})$ a morphism of diagrams of schemes, the natural morphism $\alpha^*:\DA(Q,\mathcal{Y})\to \DA(P,\mathcal{X})$ admits a right adjoint $\alpha_*:\DA(P,\mathcal{X})\to \DA(Q,\mathcal{Y})$. The existence of left adjoints is more constrained.
        
            \begin{rem}\label{rem: conventions diagrams schemes}
            Our convention for diagrams of schemes and for the variance of $\DA$ is opposed to Ayoub's but is consistent with our variance convention for derivators (see Remark \ref{rem: convention variance derivator}) and with the convention for posets of strata introduced in the next subsection.
            \end{rem}
        
    \subsection{The main theorem}\label{sec: main theorem}
    
        Let $X$ be a scheme and let $X_0$ be a dense open subscheme of $X$ with complement $Z$. We denote by $j:X_0\hookrightarrow X$ and $i:Z\hookrightarrow X$ the corresponding open and closed immersions. Let us be given a (finite) \emph{stratification} of $Z$, i.e., a finite partition of $Z$ by locally closed subschemes called \emph{strata} such that the Zariski closure of each stratum is a union of strata. The set $P$ of strata is naturally endowed with the structure of a poset where for strata $S,T\in P$ we have:
		$$S\leq T \;\;\Leftrightarrow \;\; \overline{S}\supset T \ .$$
		We thus get a stratification of $X$ indexed by the extended poset $\hat{P}=\{X_0\}\cup P$ with $X_0<S$ for all $S\in P$. 
		
		For $S\in P$ we have defined (see \S\ref{subsec:def C}) a complex of $\KM$-modules $C^\bullet(S)$ which computes the reduced cohomology groups of the poset $P_{<S}$. For strata $S,T\in P$ with $S\lessdot T$ we have defined (see \S\ref{subsec: connecting morphisms}) a morphism of complexes $b_S^T:C^\bullet(S)\to C^\bullet(T)[1]$. We also define $C^\bullet(X_0)$ to be the complex $\KM$ concentrated in degree zero. For a minimal stratum $S\in P$, i.e., such that $X_0\lessdot S$ in $\hat{P}$, we have a natural (iso)morphism of complexes $b_{X_0}^S:C^\bullet(X_0)\to C^\bullet(S)[1]$.
		
		We fix a strictly increasing map $\sigma:\hat{P}\rightarrow \mathbb{Z}$, and we assume that $\sigma(X_0)=0$. Such a map always exists. If $P$ is graded then we may take $\sigma=\rk$, the rank function.
		
		In the statement of the next theorem, we will use the following ``restriction'' morphisms of functors
		(for strata $S\leq T$):
		\begin{equation}\label{eq: res}
		\rho_S^T:
		(i_{\overline{S}}^X)_*(i_{\overline{S}}^X)^*\longrightarrow (i_{\overline{S}}^X)_*(i_{\overline{T}}^{\overline{S}})_*(i_{\overline{T}}^{\overline{S}})^*(i_{\overline{S}}^X)^* \simeq (i_{\overline{T}}^X)_*(i_{\overline{T}}^X)^*\ .
		\end{equation}
		
        \begin{thm}\label{thm:maintheorem}
        Let $\mathcal{F}\in \DA_X$ and let us set $\G=j_!j^!\mathcal{F}$.
		\begin{enumerate}[1)]
		\item There is a Postnikov system in $\DA_X$:
		$$\xymatrix{
		\cdots &\ar[rr]&&F^2\G\ar[rr]\ar[ld] && F^1\G\ar[rr]\ar[ld]  && F^0\G= \G \ar[ld]\\
		&&G^2\G\ar[lu]^{+1}&& G^1\G\ar[ul]^{+1} && G^0\G \ar[ul]^{+1}& 
		}$$
		where the graded objects are given by
		$$G^k\G = \bigoplus_{\sigma(S)=k} (i_{\overline{S}}^X)_*(i_{\overline{S}}^X)^*\F \otimes C^\bullet(S) \ .$$
		\item For every integer $k$, the connecting morphism $G^k\G\rightarrow G^{k+1}\G[1]$ has its component indexed by $S$ and $T$ with $\sigma(S)=k$, $\sigma(T)=k+1$, given by
		$$\xymatrixcolsep{4pc}\xymatrix{
		(i_{\overline{S}}^X)_*(i_{\overline{S}}^X)^*\F \otimes C^\bullet(S) \ar[r]^-{\rho_S^T\F\otimes b_S^T} & (i_{\overline{T}}^X)_*(i_{\overline{T}}^X)^*\F\otimes C^\bullet(T)[1]
		}$$
		if $S<T$ and zero otherwise.
		\item The above Postnikov system is functorial in $\F$.
		\end{enumerate}
		\end{thm}
		
		\begin{proof}
		We proceed in three steps.
		\begin{enumerate}[a)]
		\item We construct the first triangle. The (rotated) localization triangle \eqref{eq: localization triangle geometric} reads
	    $$ i_*i^*\mathcal{F}[-1]\longrightarrow j_!j^!\mathcal{F} \longrightarrow \mathcal{F} \stackrel{+1}{\longrightarrow}$$
	    and provides the first triangle of the Postnikov system, with $F^1\mathcal{G}=i_*i^*\mathcal{F}[-1]$ and $G^0\mathcal{G}=\mathcal{F}$. It is functorial in $\F$.
	    \item We work with motives over diagrams of schemes. We consider the diagram of schemes $(P,\mathcal{Z})$ where $\mathcal{Z}:P^{\mathrm{op}}\rightarrow \mathrm{Sch}$ is defined by $S\mapsto \overline{S}$ and where the transition morphisms are the natural closed immersions. We have a natural morphism of diagram of schemes $s:(P,\mathcal{Z})\rightarrow Z$ induced by the closed immersions $\overline{S}\hookrightarrow Z$. This was previously considered by Ayoub and Zucker who proved \cite[Lemma 1.18]{ayoubzucker} that the natural counit $\mathrm{id}_{\DA_Z}\rightarrow s_*s^*$ is an isomorphism. We thus have an isomorphism in $\DA_Z$: 
	    $$i_*i^*\mathcal{F}\simeq i_*s_*s^*i^*\mathcal{F}\ .$$
	    Let us recall that $(P,X)$ denotes a constant diagram of schemes. We have a natural morphism of diagrams of schemes $r:(P,\mathcal{Z})\rightarrow (P,X)$ induced by the closed immersions $\overline{S}\hookrightarrow X$. If we also denote by $p:(P,X)\rightarrow (e,X)=X$ the projection to a point, we have the following commutative diagram:
	    $$\xymatrix{
	    (P,\mathcal{Z}) \ar[r]^r \ar[d]_s & (P,X) \ar[d]^p \\
	    Z \ar[r]_i & X
	    }$$
	    We thus have an isomorphism
	    $$F^1\mathcal{G}\simeq p_*\mathcal{H}[-1]$$
	    where we set $\mathcal{H}=r_*r^*p^*\mathcal{F} \in \DA(P,X)=\DA_X(P)$.
	    It is easy to see, using the axiom DerAlg 3d in \cite[D\'{e}finition 2.4.12]{ayoubPhD1}, that the value of $\mathcal{H}$ at a stratum $S$ is $(i_{\overline{S}}^X)_*(i_{\overline{S}}^X)^*\F$. Moreover, for strata $S\leq T$ the transition map from the value at $S$ to the value at $T$ is the restriction morphism $\rho_S^T\F$ defined in \eqref{eq: res}.
	    \item We construct the Postnikov system. By applying Proposition \ref{prop: pushed} 1) to the object $\mathcal{H}\in \DA_X(P)$ we get a Postnikov system in $\DA_X$:
	    $$\xymatrix{
	    \cdots\ar[rr] && F^2p_*\mathcal{H} \ar[rr]\ar[ld]  && F^1p_*\mathcal{H}=p_*\mathcal{H}= F^1\G[1] \ar[ld]\\
		& G^2p_*\mathcal{H}\ar[ul]^{+1} && G^1p_*\mathcal{H} \ar[ul]^{+1}& 
		}$$
		with 
		$$G^kp_*\mathcal{H} \simeq \bigoplus_{\sigma(S)=k} (i_{\overline{S}}^X)_*(i_{\overline{S}}^X)^*\F\otimes C^{\bullet+1}(S)\ .$$
		This is, up to a shift, the remainder of the Postnikov system promised in the theorem, i.e., we set, for $k\geq 1$,
		$$F^k\G = F^kp_*\mathcal{H}[-1] \quad \mbox{and} \quad G^k\G = G^kp_*\mathcal{H}[-1]\ .$$
		The description of the connecting morphisms follows from Proposition \ref{prop: pushed} 2). (The connecting morphism $G^0\F\to G^1\F[1]$ needs to be treated separately; it is the composite $\F\to i_*i^*\F\to \bigoplus_{\sigma(S)=1}(i_{\overline{S}}^X)_*(i_{\overline{S}}^X)^*\F$ which is the sum of the morphisms $\rho_{X_0}^S\F$.) The functoriality statement follows from Proposition \ref{prop: pushed} 3).
        \end{enumerate}
		\end{proof}
		
		For any ($B$-)scheme $X$ let us denote by $a_X:X\to B$ its structural map. The next corollary expresses the ``compactly supported cohomology'' of a motivic sheaf $\F$ on the open $X_0$ in terms of ``compactly supported cohomology'' of $\F$ on all the closures of strata.
		
		\begin{coro}\label{coro: postnikov system base}
		Let $\F\in \DA_X$ and let us set $M=(a_{X_0})_!j^!\F\in \DA_B$.
		\begin{enumerate}[1)]
		\item 
		There is a Postnikov system in $\DA_B$:
		$$\xymatrix{
		\cdots &\ar[rr]&&F^2M\ar[rr]\ar[ld] && F^1M\ar[rr]\ar[ld]  && F^0M= M \ar[ld]\\
		&&G^2M\ar[lu]^{+1}&& G^1M\ar[ul]^{+1} && G^0M \ar[ul]^{+1}& 
		}$$
		where the graded objects are given by
		$$G^kM= \bigoplus_{\sigma(S)=k} (a_{\overline{S}})_!(i_{\overline{S}}^X)^*\F \otimes C^\bullet(S) \ .$$
		\item For every integer $k$, the connecting morphism $G^kM\rightarrow G^{k+1}M[1]$ has its component indexed by $S$ and $T$ with $\sigma(S)=k$, $\sigma(T)=k+1$, given by
		$$\xymatrixcolsep{4pc}\xymatrix{
		(a_{\overline{S}})_!(i_{\overline{S}}^X)^*\F \otimes C^\bullet(S) \ar[r]^-{\rho_S^T\F\otimes b_S^T} & (a_{\overline{T}})_!(i_{\overline{T}}^X)^*\F\otimes C^\bullet(T)[1]
		}$$
		if $S<T$ and zero otherwise.
		\item The above Postnikov system is functorial in $\F$.
		\end{enumerate}
		\end{coro}
		
		\begin{proof}
		This follows from applying the functor $(a_X)_!$ to the Postnikov system of Theorem \ref{thm:maintheorem}. By the projection formula we have an isomorphism
		$$(a_X)_!\left((i_{\overline{S}}^X)_*(i_{\overline{S}}^X)^*\F\otimes C^\bullet(S)\right) = (a_X)_!\left((i_{\overline{S}}^X)_*(i_{\overline{S}}^X)^*\F\otimes (a_X)^*C^\bullet(S)\right)\simeq (a_X)_!(i_{\overline{S}}^X)_*(i_{\overline{S}}^X)^*\F\otimes C^\bullet(S)$$ 
		and this equals $(a_{\overline{S}})_!(i_{\overline{S}}^X)^*\F\otimes C^\bullet(S)$ since $(a_X)_!(i_{\overline{S}}^X)_*=(a_X)_!(i_{\overline{S}}^X)_! = (a_{\overline{S}})_!$.
		\end{proof}
		
		\begin{rem}\label{rem: postnikov system relative cohomology}
		One can also apply the functor $(a_X)_*$ to the Postnikov system of Theorem \ref{thm:maintheorem} and get a Postnikov system expressing the relative motive of the pair $(X,Z)$ with coefficients in a motivic sheaf $\F$. It is a motivic refinement of the classical long exact sequence in relative cohomology.
		\end{rem}
		
	\subsection{Localization spectral sequences}\label{subsec: realizations}
	
	    We recover the spectral sequences of \cite{petersen} by applying realization functors.
	
	    \subsubsection{Betti realization}
	    
	        We now consider a finite type scheme $X$ over $\mathbb{C}$. We have the Betti realization functor \cite{ayoubbetti}
	        $$\DA_X \longrightarrow D(X^{\mathrm{an}})\ ,$$
	        whose target is the derived category of the category of sheaves of $\KM$-modules on the analytification $X^{\mathrm{an}}$. This functor is compatible with the operations $f^*$, $f_*$, $f_!$, $\otimes$, and we thus get from Theorem \ref{thm:maintheorem} (resp. Corollary \ref{coro: postnikov system base}) a Postnikov system in $D(X^{\mathrm{an}})$ (resp. in $D(B^{\mathrm{an}})$). We can then derive a spectral sequence by applying a cohomological functor such as the ``cohomology sheaves'' functor $\mathcal{H}^0:D(B^{\mathrm{an}})\to \operatorname{Sh}(B^{\mathrm{an}})$.
	        
	        \begin{rem}
	        We may also apply other natural cohomological functors when available. For instance, if the Betti realization of $\F$ is a complex of sheaves with constructible cohomology sheaves, almost all of which are zero (e.g., if $\F$ is a constant sheaf), then one can also apply the perverse cohomology functor ${}^{p}\mathrm{H}^0$ with target the category of perverse sheaves ${}^p\mathrm{Perv}(B^{\mathrm{an}})$ for any perversity function $p$ \cite{bbd}.
	        \end{rem}
	        
	        In the case $B=\operatorname{Spec}(\mathbb{C})$, the spectral sequence reads:
	        \[
	        E_1^{p,q} = \bigoplus_{\sigma(S)=p}
	        H^{p+q}\big(R\Gamma_c(i_{\overline{S}}^X)^*\F\otimes C^\bullet(S)\big)
	        \quad \Rightarrow \quad H^{p+q}_c(X_0,j^!\F)\ .
	        \]
	        We can make it more explicit under some extra assumptions as in \cite[\S 3]{petersen}, and we get for instance the following corollary \cite[Theorem 3.3 (ii)]{petersen}. We recall the notation $h^n(S)=H^n(C^\bullet(S))$ from \S\ref{subsec:def C}.
	    
	        \begin{coro}\label{coro: petersen}
	        Assume that $\KM$ is a hereditary ring (e.g., $\KM$ is a field or $\KM=\mathbb{Z}$) and that for every stratum $S$ and every integer $n$ the cohomology group $h^n(S)$ is a torsion-free $\KM$-module. Then we have a spectral sequence of $\KM$-modules:
	        \[
	        E_1^{p,q} = \bigoplus_{\substack{\sigma(S)=p\\i+j=p+q}}
	        H^{i}_c\big(\overline{S},(i_{\overline{S}}^X)^*\F\big) \otimes h^j(S)
	        \quad \Rightarrow \quad H^{p+q}_c(X_0,j^!\F)\ .
	        \]
	        \end{coro}
	        
	        \begin{proof}
	        Since $C^\bullet(S)$ is a complex of free $\KM$-modules, the tensor product by $C^\bullet(S)$ is also the derived tensor product. Moreover, since $\KM$ is hereditary, the complex $C^\bullet(S)$ is quasi-isomorphic to its cohomology. Finally, since that cohomology is assumed to be torsion-free, the K\"{u}nneth formula applies without the Tor correction term.
	        \end{proof}
	        
	        \begin{rem}\label{rem: coro petersen CM}
	        In the context of Remark \ref{rem: graded poset OS} we can simplify further since most cohomology groups $h^j(S)$ vanish: we get a spectral sequence
	        $$E_1^{p,q}= \bigoplus_{\rk(S)=p}H^{q}_c\big(\overline{S},(i_{\overline{S}}^X)^*\F\big) \otimes h(S)^\vee \quad \Rightarrow \quad H^{p+q}_c(X_0,j^!\F)\ .$$
	        The differential $d_1^{p,q}$ has component indexed by strata $S$ and $T$, with $\rk(S)=p$, $\rk(T)=p+1$, given by
	        $$\xymatrixcolsep{4pc}\xymatrix{ H^{q}_c\big(\overline{S},(i_{\overline{S}}^X)^*\F\big) \otimes h(S)^\vee \ar[r]^{\rho_S^T\F\otimes b_S^T} & H^{q}_c\big(\overline{T},(i_{\overline{T}}^X)^*\F\big) \otimes h(T)^\vee }$$
	        if $S<T$, and zero otherwise.
	        \end{rem}
	        
	   \subsubsection{Hodge realization}
	        
	        In the case $\KM=\mathbb{Q}$, the Betti realization functor can be enriched into a Hodge realization functor in the constructible case. Following \cite[Definition 2.11]{ayoubguide} we define $\DA^{\ct}_X$ to be the smallest triangulated subcategory of $\DA_X$ stable under direct summands and Tate twists and containing the motives $f_*\KM_Y$ for $f:Y\to X$ of finite presentation. Objects of $\DA^{\ct}_X$ are called \emph{constructible}.
	        
	        Thanks to \cite{ivorrarealization} we have Hodge realization functors
	        $$\DA^{\mathrm{ct}}_X \longrightarrow D^b(\operatorname{MHM}(X))$$
	        which are compatible with the six functor formalism, where $\operatorname{MHM}(X)$ is Saito's category of mixed Hodge modules on $X$ \cite{saitoMHM}.  This proves that the spectral sequence of Corollary \ref{coro: petersen} is compatible with mixed Hodge structures if $X$ has finite type over $\operatorname{Spec}(\mathbb{C})$ and $\F$ is constructible, e.g., $\F=\mathbb{Q}_X$ the constant sheaf. This was already noted by Petersen \cite[Theorem 3.3 (ii)]{petersen}.
	    
	\subsubsection{\'{E}tale (and $\ell$-adic) realization}
	    
	    Let us assume that $B=\operatorname{Spec}(k)$ for some field $k$. We fix a prime $\ell$ invertible in $k$ and set $\KM=\QM_\ell$. By \cite[\S 5 and \S 9]{ayoubetale} and \cite[\S 7.2]{cisinskidegliseetale}, we have an \'{e}tale (or $\ell$-adic) realization functor
	    \[
	    \DA^{\mathrm{ct}}_X \longrightarrow D^b_c(X^{\text{\'{e}t}})
	    \]
	    compatible with the six operations, where $D^b_c(X^{\text{\'{e}t}})$ is Ekedahl's triangulated category of $\ell$-adic systems \cite{ekedahl}.
	    
	    This implies that we have a spectral sequence in \'{e}tale cohomology analogous to that of Corollary \ref{coro: petersen} with $\QM_\ell$ coefficients, with values in the category of continuous representations of the Galois group $\operatorname{Gal}(k^\mathrm{sep}/k)$. This was already noted by Petersen \cite[Theorem 3.3 (ii)]{petersen}.

	\subsection{The dual version}\label{subsec: dual version}
	    
	    We start with the ``dual'' variant of Theorem \ref{thm:maintheorem}, where we consider the same geometric situation but study the object $j_*j^*\F$ instead of $j_!j^!\F$. We will derive one from the other by using Verdier duality in the motivic setting (see Remark \ref{rem: duality} below for a discussion of this strategy). 
	    
	    For simplicity we assume that the base scheme $B$ is of finite type over a characteristic zero field. Then we have a Verdier duality functor \cite[Theorem 3.10]{ayoubguide}
	    $$\D_X:(\DA^{\ct}_X)^{\mathrm{op}}\longrightarrow \DA^{\ct}_X$$
	    which satisfies the usual compatibilities $\D_X\circ \D_X\simeq \id$ and $\D_Y\circ f_*\simeq f_!\circ \D_X$ for $f:X\to Y$ a morphism of schemes.
	    
	    Recall from \S\S\ref{subsec:def C} and \ref{subsec: connecting morphisms} the homological complexes $C_\bullet(S)$, for $S\in P$, that we now treat with cohomological conventions (i.e., with negative cohomological degrees) and the connecting morphisms $b_S^T:C_{\bullet+1}(T)\to C_\bullet(T)$ for $S\lessdot T$, which in cohomological conventions read: $b_S^T:C_\bullet(T)\to C_\bullet(S)[1]$. As in the previous paragraph we set $C_\bullet(X_0)=\KM$ concentrated in degree $0$, and for $S\in P$ a minimal element, we consider the natural (iso)morphism $b_{X_0}^S:C_\bullet(S)\to C_\bullet(X_0)[1]$.
	    
	    In the statement of the next theorem we will use the following ``Gysin-type'' morphisms of functors, which are dual to restriction morphisms $\rho_S^T$ (for strata $S\leq T$):
	    \begin{equation}\label{eq: gysin}
	    \gamma_S^T:
	    (i_{\overline{T}}^X)_!(i_{\overline{T}}^X)^! \simeq (i_{\overline{S}}^X)_!(i_{\overline{T}}^{\overline{S}})_!
	    (i_{\overline{T}}^{\overline{S}})^!(i_{\overline{S}}^X)^! \longrightarrow (i_{\overline{S}}^X)_!(i_{\overline{S}}^X)^!\ .
	    \end{equation}
	    
	    \begin{thm}\label{thm: main theorem dual}
	    Let $\mathcal{F}\in \DA^{\ct}_X$ be a constructible object and let us set $\G=j_*j^*\mathcal{F}$.
		\begin{enumerate}[1)]
		\item There is a Postnikov system in $\DA_X$:
		$$\xymatrix{
		\G=F_0\G \ar[rr] && F_1\G \ar[dl]^{+1}\ar[rr] && F_2\G \ar[dl]^{+1}\ar[rr] && \ar[dl]^{+1} & \cdots \\
		& G_0\G \ar[ul] && G_1\G \ar[ul]&& G_2\G \ar[ul]&&&
		}$$
		where the graded objects are given by
		$$G_k\G = \bigoplus_{\sigma(S)=k} (i_{\overline{S}}^X)_!(i_{\overline{S}}^X)^!\F \otimes C_\bullet(S) \ .$$
		\item  For every integer $k$, the connecting morphism $G_{k+1}\G\rightarrow G_k\G[1]$ has its component indexed by $S$ and $T$ with $\sigma(S)=k$, $\sigma(T)=k+1$, given by
		$$\xymatrixcolsep{4pc}\xymatrix{
		(i_{\overline{T}}^X)_!(i_{\overline{T}}^X)^!\F\otimes C_\bullet(T)
		 \ar[r]^-{\gamma_S^T\F\otimes b_S^T} &  (i_{\overline{S}}^X)_!(i_{\overline{S}}^X)^!\F \otimes C_\bullet(S)[1]
		}$$
		if $S<T$, and zero otherwise.
		\item The above Postnikov system is functorial in $\F$.
		\end{enumerate}
	    \end{thm}
	    
	    \begin{proof}
	    We apply Theorem \ref{thm:maintheorem} to the Verdier dual of $\F$ and dualize the Postnikov system obtained in this way. The only thing that needs to be checked is the description of $G_k\G$ and the connecting morphisms. Let $\omega_X\in \DA_X^{\ct}$ denote the dualizing object. For any object $\mathcal{U}\in \DA^{\ct}_X$ we have:
	    $$\D_X(\mathcal{U}\otimes C^\bullet(S)) = \underline{\Hom}_{\DA^{\ct}_X}(C^\bullet(S)\otimes \mathcal{U},\omega_X) \simeq \underline{\Hom}_{\DA^{\ct}_X}(C^\bullet(S),\D_X\mathcal{U})\simeq \D_X\mathcal{U}\otimes C_\bullet(S)\ .$$
	    In the last step we have used the fact that $C_\bullet(S)$ is the strong dual of $C^\bullet(S)$ in the monoidal category $\DM_{\KMMod}$ because it is a bounded complex of free $\KM$-modules of finite rank. By applying this to $\mathcal{U}=(i_{\overline{S}}^X)_*(i_{\overline{S}}^X)^*\D_X\F$, using the compatibility between Verdier duality and the functors $i_*$ and $i_!$, and the fact that $\D_X\circ \D_X\F\simeq \F$, we get an isomorphism: 
	    $$\D_X\left((i_{\overline{S}}^X)_*(i_{\overline{S}}^X)^*\D_X\F\otimes C^\bullet(S)\right)\simeq (i_{\overline{S}}^X)_!(i_{\overline{S}}^X)^!\F\otimes C_\bullet(S)\ .$$
	    This implies the description of $G_k\G$ as in the statement of the theorem. The fact that the Gysin morphisms $\gamma_S^T$ defined in \eqref{eq: gysin} and the restriction morphisms $\rho_S^T$ defined in \eqref{eq: res} are Verdier dual to each other is clear, and the claim follows.
	    \end{proof}
	    
	    \begin{rem}\label{rem: duality}
	    Theorem \ref{thm: main theorem dual} is most certainly true without the assumption that $\F$ is constructible and without the assumption that $B$ is a finite type scheme over a characteristic zero field. In fact, as noted in the introduction, we can prove it without the functoriality statement using only the language of triangulated categories. However, it seems that the tools that we are using do not allow us to do it functorially. Indeed, we cannot simply repeat the proof of Theorem \ref{thm:maintheorem} since the existence of a left adjoint to the functor $s^*$ appearing in the proof is not guaranteed in the context of an algebraic derivator.
	    \end{rem}
	    
	    \begin{rem}
	    As in Corollary \ref{coro: postnikov system base} and Remark \ref{rem: postnikov system relative cohomology} one may apply the functors $(a_X)_*$ or $(a_X)_!$ to the Postnikov system of Theorem \ref{thm: main theorem dual} to get localization Postnikov systems in $\DA_B$. In the case of $(a_X)_*$ this computes $(a_{X_0})_*j^*\F$, the cohomology of $X_0$ with coefficients in the restriction of $\F$; a particularly interesting case is when $\F=\KM_X$ is a constant motivic sheaf. There the main difficulty is to be able to compute the graded objects of the Postnikov system, i.e., the objects $(a_{\overline{S}})_*(i_{\overline{S}}^X)^!\KM_X$ for all strata $S$. Luckily, if $\overline{S}$ is smooth of codimension $c$ in $X$, then by purity we have an isomorphism
	    $$(i_{\overline{S}}^X)^!\KM_X \simeq \KM_{\overline{S}}[-2c](-c)\ ,$$
	    and the localization Postnikov system is expressed in terms of the motives of the closures of strata. 
	    \end{rem}
	    
	    \begin{rem}
	    By applying realization functors and cohomological functors one gets spectral sequences from Theorem \ref{thm: main theorem dual} as in \S \ref{subsec: realizations}. We only state one special case that is important for applications. Let $\F=\KM_X$, and assume that we are in the context of Corollary \ref{coro: petersen} and Remark \ref{rem: coro petersen CM}. Further assume that for every stratum $S$ the closure $\overline{S}$ is smooth of codimension $c_S$ in $X$. Then we get by the previous remark a (second quadrant) spectral sequence in mixed Hodge structures:
	    \begin{equation}\label{eq: spectral sequence classical}
	    E_1^{-p,q} = \bigoplus_{\rk(S)=p}H^{q-2c_S}(\overline{S})(-c_S)\otimes h(S) \quad \Rightarrow \quad H^{-p+q}(X_0)\ .
	    \end{equation}
	    A special case of interest is when the stratification is induced by a normal crossing divisor, in which case $c_S=\rk(S)$ and $h(S)$ has rank one for every stratum $S$; one then recovers Deligne's spectral sequence \cite[3.2.4.1]{delignehodgeII}. The other classical spectral sequences cited in the introduction \cite{cohentaylor, goreskymacpherson, looijenga, kriz, totaro, bjornerekedahl, getzler, dupontOS, bibby} are all special cases of \eqref{eq: spectral sequence classical}.
	    \end{rem}
	
    \subsection{Functoriality}
    
        We now turn to the functoriality of our main theorem with respect to morphisms of schemes. With a little more work it should be easy to treat more general cases.
    
        \subsubsection{A category of stratified schemes}
        
            For simplicity we restrict to morphisms between stratified schemes whose underlying combinatorial datum is an isomorphism of posets.
        
            \begin{defi}
            Let $X$ and $X'$ be two stratified schemes with posets of strata $\hat{P}$ and $\hat{P}'$ as in \S\ref{sec: main theorem}. A \emph{stratified morphism} from $X$ to $X'$ is a pair $(\alpha,f)$ where $\alpha:\hat{P}\to \hat{P}'$ is an isomorphism of posets and $f:X\to X'$ is a morphism of schemes such that 
            $$\forall S\in \hat{P}\; , \; f(\overline{S})\subset \overline{\alpha(S)}\ .$$
            \end{defi}
            
            Note that for a stratified morphism $(\alpha, f)$, the morphism $f$ does not determine $\alpha$ in general. However, for an isomorphism of schemes $f:X\to X'$ such that the image by $f$ of every stratum of $X$ is a stratum of $X'$, there is a unique $\alpha:\hat{P}\to \hat{P}'$ such that $(\alpha,f)$ is a stratified isomorphism.
            
            Our notion of stratified morphism is more easily understood in the context of the category of diagrams of schemes. For a stratified scheme $X$ with poset of strata $\hat{P}$ we have a natural diagram of schemes $(\hat{P},\mathcal{X})$ where $\mathcal{X}:\hat{P}^{\mathrm{op}}\to \Sch$ sends $S$ to $\overline{S}$. A stratified morphism $(\alpha,f)$ as above gives rise to a morphism of diagrams of schemes
            $$(\alpha,f):(\hat{P},\mathcal{X})\longrightarrow (\hat{P}',\mathcal{X}')\ .$$
            One can thus view our category of stratified schemes as a subcategory of the category of diagrams of schemes. It is not a full subcategory since we only consider morphisms $(\alpha,f)$ for which $\alpha$ is an isomorphism of posets.
            
        \subsubsection{Functoriality of the localization triangle}\label{subsubsec: functoriality localization triangle}
        
        The first step in the construction of the Postnikov system is just the localization triangle \eqref{eq: localization triangle geometric}. So let us consider a morphism of pairs $f:(X,Z)\to (X',Z')$, where $Z$ and $Z'$ are closed subschemes and $f(Z) \subset Z'$. If we denote by $X_0$ and $X'_0$ the open complements, then $f^{-1}(X'_0) \subset X_0$. We have the following diagram, where the left square is commutative and the rectangle on the right is cartesian.
           \[
	        \xymatrix{
	        Z \ar[r]^{i} \ar[d]_f &
	        X \ar[d]_f &
	        X_0 \ar[l]_j &
	        f^{-1}(X'_0) \ar[l]_(.6){j_0} \ar[d]^f \\
	        Z' \ar[r]_{i'}&
	        X' &&
	        X'_0 \ar[ll]^{j'} 
	        }
	        \]
	        
	        Given an object $\F'\in \DA_{X'}$, we want to define a morphism between the localization triangle for $\F'$ and $f_*$ of the localization triangle for $f^*\F'$:
	        
	        \[
	        \xymatrix{
	        (i')_*(i')^*\F'[-1]\ar[r]\ar[d] &
	        (j')_!(j')^!\F'\ar[r]\ar[d] &
	        \F'\ar[r]^(.7){+1}\ar[d] & \\
	        f_*i_*i^*f^*\F'[-1]\ar[r] &
	        f_*j_!j^!f^*\F'\ar[r] &
	        f_*f^*\F'\ar[r]^(.7){+1} & \\
	        }
	        \]
	        
	        Let us now define the three vertical morphisms:
	        
	        \begin{itemize}
	        \item 
	        The right morphism is of course the adjunction unit $\F' \to f_*f^*\F'$.
	        
	        \item
	        The left morphism is given by the composition:
	        \[
	        (i')_*(i')^*\F'[-1] \longrightarrow
	        (i')_*f_*f^*(i')^*\F'[-1] \stackrel\sim\longrightarrow
	        f_*i_*i^*f^*\F'[-1]
	        \]
	        where the first arrow is induced by the adjunction unit, and the isomorphism on the right follows from the commutativity of the left square in the diagram above.
	        
	        \item
	        The middle morphism is given by the composition:
	        \[
	        (j')_!(j')^!\F' \longrightarrow
	        (j')_!f_*f^*(j')^!\F' \longrightarrow
	        f_*j_!(j_0)_!(j_0)^!j^!f^*\F' \longrightarrow
	        f_*j_!j^!f^*\F'
	        \]
	        where the first arrow is induced by the adjunction unit, the second arrow induced by two exchange morphisms (which are part of the cross functor structure, see \cite[\S 1.2]{ayoubPhD1}) for the cartesian square on the right of the diagram above, and the third arrow is induced by the adjunction counit.
	        \end{itemize}
	        
	        We leave it to the reader to check that this defines indeed a morphism of triangles. The commutativity of the left square is easy, the commutativity of the right square is a nice exercise on using the axioms of a cross functor, and the commutativity of the third square follows from \cite[Proposition 1.1.9]{bbd}.
	        
	        \begin{rem}
	        Assume that $B=\operatorname{Spec}(\mathbb{C})$ and denote by $a: X \to B$ and $a':X'\to B$ the structure morphisms. If $f$ is proper, we have $a'_!f_* = a'_!f_! = a_!$. Consequently, taking $\F' = \mathbb{Q}_{X'}$, applying the functor $a'_!$ and taking the Betti realization, we get the functoriality (for proper morphisms) of the localization long exact sequence of the introduction:
	        \[
	       \xymatrix{
	       \cdots \ar[r] &
	       H^\bullet_c(X'_0) \ar[r]\ar[d] &
	       H^\bullet_c(X') \ar[r]\ar[d] &
	       H^\bullet_c(Z') \ar[r]\ar[d] &
	       H^{\bullet+1}_c(X'_0)  \ar[r]\ar[d] &
	       \cdots\\
	       \cdots \ar[r] &
	       H^\bullet_c(X_0) \ar[r] &
	       H^\bullet_c(X) \ar[r] &
	       H^\bullet_c(Z) \ar[r] &
	       H^{\bullet+1}_c(X_0)  \ar[r] &
	       \cdots
	       }
           \]
	        
	        Similarly, using $a'_*$ instead, we get the functoriality of the long exact sequence in relative cohomology:
	       \[
	       \xymatrix{
	       \cdots \ar[r] &
	       H^\bullet(X',Z') \ar[r]\ar[d] &
	       H^\bullet(X') \ar[r]\ar[d] &
	       H^\bullet(Z') \ar[r]\ar[d] &
	       H^{\bullet+1}(X',Z')  \ar[r]\ar[d] &
	       \cdots\\
	       \cdots \ar[r] &
	       H^\bullet(X,Z) \ar[r] &
	       H^\bullet(X) \ar[r] &
	       H^\bullet(Z) \ar[r] &
	       H^{\bullet+1}(X,Z)  \ar[r] &
	       \cdots
	       }
           \]
	       In this case we do not need to assume that $f$ is proper: we always have $a'_*f_* = a_*$.
	       \end{rem}
            
        \subsubsection{Functoriality of the localization spectral sequence}
        
            To express the functoriality of Theorem \ref{thm:maintheorem} with respect to stratified morphisms, we adopt a more meaningful notation.
            \begin{enumerate}[$\bullet$]
            \item For an object $\mathcal{H}\in \DA_X(P)$ we denote by $\widetilde{\Pi}(\mathcal{H})$ the Postnikov system in $\DA_X$ described in Proposition \ref{prop: pushed}.
            \item For an object $\F\in \DA_X$ we denote by $\Pi(\hat{P},X;\F)$ the Postnikov system in $\DA_X$ described in Theorem \ref{thm:maintheorem}.
            \end{enumerate}
            Borrowing notation from the proof of Theorem \ref{thm:maintheorem} we have that $\Pi(\hat{P},X;\F)$ is obtained by appending $\widetilde{\Pi}(r_*r^*p^*\F)[-1]$ to the first (localization) triangle.

            We start with a general lemma explaining the compatibility between the Postnikov systems $\widetilde{\Pi}$ and certain pushforwards. We recall (see Remark \ref{rem: functoriality complexes C}) that an isomorphism of posets $\alpha:P\to P'$ induces isomorphisms of complexes denoted
            $$C^\bullet(\alpha):C^\bullet_{P'}(S')\to C^\bullet_{P}(S)$$ 
            for elements $S\in P$ and $S'\in P'$ such that $S'=\alpha(S)$. If $\sigma:\hat{P}\to \mathbb{Z}$ is a strictly increasing map such that $\sigma(\hat{0})=0$ and if $\alpha:\hat{P}\to \hat{P}'$ is an isomorphism of posets then we denote by $\sigma':\hat{P}'\to \mathbb{Z}$ the composite $\sigma'=\sigma\circ \alpha^{-1}$. In the next lemma, for $\mathcal{H}\in \DA_X(P)$ and $S\in P$ we denote by $\mathcal{H}_S\in \DA_X$ the value of $\mathcal{H}$ at $S$.
            
            \begin{lem}\label{lem: functoriality technical}
            Let $\alpha:P\to P'$ be an isomorphism of posets, let $f:X\to X'$ be a morphism of schemes, and let us denote by $(\alpha,f):(P,X)\to (P',X')$ the corresponding morphism of (constant) diagrams of schemes. For $\mathcal{H}\in \DA_X(P)$ we have an isomorphism:
            $$\widetilde{\Pi}((\alpha,f)_*\mathcal{H}) \stackrel{\sim}{\longrightarrow} f_*\widetilde{\Pi}(\mathcal{H})\ .$$
            At the level of graded objects it reads
            $$\bigoplus_{\sigma(S')=k}f_*\mathcal{H}_{\alpha^{-1}(S')}\otimes C_{P'}^{\bullet+1}(S') \stackrel{\sim}{\longrightarrow} \bigoplus_{\sigma(S)=k}f_*(\mathcal{H}_S\otimes C_P^{\bullet+1}(S))\simeq \bigoplus_{\sigma(S)=k}f_*\mathcal{H}_S\otimes C_P^{\bullet+1}(S)$$
            and its component indexed by $S'$ and $S$ is given by $\id\otimes C^{\bullet+1}(\alpha)$ if $S=\alpha(S')$ and zero otherwise.
            \end{lem}
            
            \begin{proof}
            Since $(\alpha,f)=(\id,f)\circ(\alpha,\id)$ it is enough to do the proof in the case $\alpha=\id$ and in the case $f=\id$. In the former case it follows from the fact that $(\id,f)_*:\DA_X\to \DA_{X'}$ is a morphism of derivators. In the latter case it is the content of Remark \ref{rem: functoriality combinatorial postnikov system}.
            \end{proof}
            
            In the statement of the next theorem we will use the following ``pullback'' morphisms of functors in the context of a morphism of schemes $f:X\to X'$ and two strata $S$ and $S'$ such that $f(\overline{S})\subset \overline{S'}$, where $f_{\overline{S}}^{\overline{S'}}:\overline{S}\to\overline{S'}$ denotes the morphism induced by $f$:
            $$\eta_S^{S'}(f):(i_{\overline{S'}}^{X'})_*(i_{\overline{S'}}^{X'})_* \longrightarrow (i_{\overline{S'}}^{X'})_*(f_{\overline{S}}^{\overline{S'}})_*(f_{\overline{S}}^{\overline{S'}})^*(i_{\overline{S'}}^{X'})^* \simeq f_*(i_{\overline{S}}^X)_*(i_{\overline{S}}^X)^*f^*\ .$$

            \begin{thm}\label{thm: functoriality}
            \phantom{}
            \begin{enumerate}[1)]
            \item The Postnikov system of Theorem \ref{thm:maintheorem} is functorial with respect to stratified morphisms. More precisely, for every morphism $(\alpha,f):(\hat{P},X)\to (\hat{P}',X')$ and every object $\F'\in \DA_{X'}$ we have a morphism of Postnikov systems
            $$\Pi(\alpha,f;\F'):\Pi(\hat{P}',X';\F') \longrightarrow f_*\Pi(\hat{P},X;f^*\F')\ .$$
            They satisfy $\Pi(\id,\id;\F')=\id$ and the equality
            $$\Pi(\beta\circ\alpha,g\circ f;\F'') = g_*\Pi(\alpha,f;g^*\F'') \circ \Pi(\beta,g;\F'')\ .$$
            for composable morphisms
            $$(\hat{P},X)\stackrel{(\alpha,f)}{\longrightarrow} (\hat{P}',X') \stackrel{(\beta,g)}{\longrightarrow} (\hat{P}'',X'')$$
            and $\F''\in \DA_{X''}$.
            \item For every integer $k$, the morphism $\Pi(\alpha,f;\F')$ reads, at the level of graded objects:
            $$\bigoplus_{\sigma'(S')=k}(i_{\overline{S'}}^{X'})_*(i_{\overline{S'}}^{X'})^*\F' \otimes C^\bullet_{P'}(S') \longrightarrow \bigoplus_{\sigma(S)=k} f_*(i_{\overline{S}}^X)_*(i_{\overline{S}}^X)^*f^*\F'\otimes C^\bullet_P(S)$$
            and has its component indexed by $S'$ and $S$ given by $\eta_S^{S'}(f)\F'\otimes C^\bullet(\alpha)$ if $S'=\alpha(S)$ and zero otherwise.
            \item The morphism $\Pi(\alpha,f;\F')$ is functorial in $\F'$.
            \end{enumerate}
            \end{thm}
	
	        \begin{proof}
	        We proceed in three steps as in the proof of Theorem \ref{thm:maintheorem}.
	        \begin{enumerate}[a)]
	        \item The first triangle of the Postnikov system is the localization triangle and its functoriality follows from the discussion of \S \ref{subsubsec: functoriality localization triangle}.
	        \item Following the proof of Theorem \ref{thm:maintheorem} we consider the following commutative diagram in the category of diagrams of schemes.
	        \[
            \xymatrix{
            ( P, \mathcal{Z} ) \ar[r]^r \ar[d]_s \ar@/^/[drr]^(.7){(\alpha, f)} &
            ( P, X ) \ar[d]_p \ar@/^/[drr]^(.7){(\alpha, f)} \\
            Z \ar[r]^i \ar@/^/[drr]^(.7){f} & X \ar@/^/[drr]^(.7){f}&
            ( P', \mathcal{Z'} ) \ar[r]^{r'} \ar[d]_{s'} &
            ( P', X' ) \ar[d]^{p'} \\
            && Z' \ar[r]_{i'} & X'
            }
            \]
            The morphism $(\alpha,f):(P,X)\to (P,X')$ is induced by $\alpha$ at the level of posets and by $f:X\to X'$ at the level of schemes. The morphism $(\alpha,f):(P,\mathcal{Z})\to (P,\mathcal{Z}')$ is induced by $\alpha$ at the level of posets and by the maps $\overline{S}\to \overline{\alpha(S)}$ induced by $f$ at the level of schemes. We have the following commutative diagram in $\DA_{X'}$, where the vertical arrows $\stackrel{\sim}{\to}$ are isomorphisms by \cite[Lemma 1.18]{ayoubzucker} as in the proof of Theorem \ref{thm:maintheorem}.
            $$\xymatrix{
            (i')_*(i')^*\F' \ar[rr] \ar[d]_{\sim} && f_*i_*i^*f^*\F' \ar[d]^\sim \\
            (i')_*(s')_*(s')^*(i')^*\F' \ar@{<->}[d]_\sim && f_*i_*s_*s^*i^*f^*\F' \ar@{<->}[d]^\sim \\
            (p')_*(r')_*(r')^*(p')^*\F' \ar[r]_-{(p')_*\varphi} & (p')_*(\alpha,f)_*r_*r^*(\alpha,f)^*(p')^*\F' \ar@{<->}[r]_-{\sim} & f_*p_*r_*r^*p^*f^*\F'
            }$$
            We have the objects 
            $$\mathcal{H}'=(r')_*(r')^*(p')^*\F' \quad \mbox{ and } \quad \mathcal{H} = r_*r^*(\alpha,f)^*(p')^*\F'\simeq r_*r^*p^*f^*\F'$$
            of $\DA_{X'}(P')$ and $\DA_X(P)$ respectively, and the natural morphism $\varphi:\mathcal{H}'\to (\alpha,f)_*\mathcal{H}$ appearing in the above diagram. For $S'\in P'$, the value of $\mathcal{H}'$ at $S'$ is $(i_{\overline{S'}}^{X'})_*(i_{\overline{S'}}^{X'})^*\F'$, that of $(\alpha,f)_*\mathcal{H}$ is $f_*(i_{\overline{S}}^X)_*(i_{\overline{S}}^X)^*f^*\F'$, for $S'=\alpha(S)$, and the value of $\varphi$ is $\eta_S^{S'}(f)\F'$. 
            \item We define the remainder of $\Pi(\alpha,f;\F')$ to be the composite
            $$\xymatrixcolsep{2.5pc}\xymatrix{ \widetilde{\Pi}(\mathcal{H}') \ar[r]^-{\widetilde{\Pi}(\varphi)} & \widetilde{\Pi}((\alpha,f)_*\mathcal{H})   \ar[r]^-{\sim} & f_*\widetilde{\Pi}(\mathcal{H}) }$$
            where the second arrow is described in Lemma \ref{lem: functoriality technical}. The compatibility with composition is left to the reader. The description of $\Pi(\alpha,f;\F')$ at the level of graded objects follows from Lemma \ref{lem: functoriality technical} and the description of the values of $\varphi$ in b). The functoriality in $\F'$ is obvious.
	        \end{enumerate}
	        \end{proof}
	        
	        \begin{rem}
	        By applying Poincar\'{e}--Verdier duality one gets the dual statement that the Postnikov system of Theorem \ref{thm: main theorem dual} is functorial with respect to stratified morphisms.
	        \end{rem}

\bibliographystyle{alpha}

\bibliography{biblio}

\end{document}